\documentclass[11pt,preprint]{imsart}
\usepackage{amsmath,amsthm,verbatim,amssymb,amsfonts,amscd, graphicx}
\usepackage{amssymb,bm}
\usepackage{mathrsfs}
\usepackage{hyperref}
\hypersetup{
    colorlinks,
    citecolor=black,
    filecolor=black,
    linkcolor=blue,
    urlcolor=black
}
\usepackage{url}
\usepackage{verbatim}
\RequirePackage{natbib}
\usepackage{constants}
\usepackage{enumerate}
\newcommand{\beas}{\begin{eqnarray*}}
\newcommand{\enas}{\end{eqnarray*}}
\newcommand{\bea}{\begin{eqnarray}}
\newcommand{\ena}{\end{eqnarray}}
\newcommand{\bms}{\begin{multline*}}
\newcommand{\ems}{\end{multline*}}
\newcommand{\bels}{\begin{align*}}
\newcommand{\enls}{\end{align*}}
\newcommand{\bel}{\begin{align}}
\newcommand{\enl}{\end{align}}

\newcommand{\ignore}[1]{}
\newcommand{\tr}{\mbox{tr}}

\newtheorem{theorem}{Theorem}[section]
\newtheorem{corollary}{Corollary}[section]

\newtheorem{proposition}{Proposition}[section]
\newtheorem{remark}{Remark}[section]
\newtheorem{lemma}{Lemma}[section]
\newtheorem{definition}{Definition}[section]

\def\blfootnote{\xdef\@thefnmark{}\@footnotetext}
\newcommand{\expect}[1]{\mathbb{E}{\l(#1\r)}}
\newcommand{\mf}[1]{\mathbf{#1}}

\newcommand{\dotp}[2]{\left\langle#1,#2\right\rangle}
\newcommand{\m}{\mathcal}
\newcommand{\mb}{\mathbb}
\newcommand\argmin{\mathop{\mbox{argmin}}}

\newcommand{\sign}{\mathrm{sign}}
\def\r{\right}
\def\l{\left}
\newcommand{\eps}{\varepsilon}
\newcommand{\var}{\mbox{Var}}

\begin{document}

\begin{frontmatter}
\title{
%LASSO and single index model with heavy-tailed measurements
Structured signal recovery from non-linear and heavy-tailed measurements
}
\runtitle{Recovery from non-linear and heavy-tailed measurements}
\begin{aug}
\author{\fnms{Larry} \snm{Goldstein}\thanksref{t2,t3}\ead[label=e1]{larry@usc.edu}},
\author{\fnms{Stanislav} \snm{Minsker}\thanksref{t2}\ead[label=e2]{minsker@usc.edu}}
\and
\author{\fnms{Xiaohan} \snm{Wei}\thanksref{t1}\ead[label=e3]{xiaohanw@usc.edu}}
 \thankstext{t2}{Department of Mathematics, University of Southern California}
  \thankstext{t1}{Department of Electrical Engineering, University of Southern California}
  \thankstext{t3}{Larry Goldstein was partially supported by NSA grant H98230-15-1-0250.}
\runauthor{L. Goldstein, S. Minsker, X. Wei}

\affiliation{University of Southern California}
\printead{e1,e2,e3}
\end{aug}
\maketitle
%\maketitle
\begin{abstract}
We study high-dimensional signal recovery from non-linear measurements with design vectors having elliptically symmetric distribution. 
Special attention is devoted to the situation when the unknown signal belongs to a set of low statistical complexity, while both the measurements and the design vectors are heavy-tailed. 
We propose and analyze a new estimator that adapts to the structure of the problem, while being robust both to the possible model misspecification characterized by arbitrary non-linearity of the measurements as well as to data corruption modeled by the heavy-tailed distributions. Moreover, this estimator has low computational complexity. 
Our results are expressed in the form of exponential concentration inequalities for the error of the proposed estimator. 
On the technical side, our proofs rely on the generic chaining methods, and illustrate the power of this approach for statistical applications.  
Theory is supported by numerical experiments demonstrating that our estimator outperforms existing alternatives when data is heavy-tailed. 
\end{abstract}


\end{frontmatter}

%###############
\section{Introduction.}
%###############

%#########################
%\subsection{Problem Formulation.}
%#########################

%\marginpar{\Stas{Notation that we are using is inherited from Vershynin but is a bit ``uncommon''. I suggest that we transition to something like $y_i=f(\dotp{x_i}{\theta_\ast},\eps_i)$ since it will make reading easier for those who are used to reading papers about Lasso.}}
Let $(\mf{x},y)\in \mb R^d\times \mb R$ be a random couple with distribution $P$ governed by the \textit{semi-parametric single index model}
\begin{equation}
\label{model}
y=f(\langle\mathbf{x},\theta_*\rangle,\delta),
\end{equation}
where $\mf{x}$ is a measurement vector with marginal distribution $\Pi$, $\delta$ is a noise variable that is assumed to be independent of $\mf{x}$, $\theta_\ast \in \mb R^d$ is a fixed but otherwise unknown signal (``index vector''), and $f:\mathbb{R}^2\mapsto\mathbb{R}$ is an unknown link function; here and in what follows, $\dotp{\cdot}{\cdot}$ denotes the Euclidean dot product.  
We impose no explicit conditions on $f$, and in particular it is not assumed that $f$ is convex, or even continuous.
%Let $(\mf{x_1},y_1),\ldots,(\mf{x_m},y_m)$ be the training data - a sequence of i.i.d. copies of $(\mf{x},y)$. 
Our goal is to estimate the signal $\theta_\ast$ from the training data $(\mf{x_1},y_1),\ldots,(\mf{x_m},y_m)$ - a sequence of i.i.d. copies of $(\mf{x},y)$ defined on a probability space $\l(\Omega,\m B, \mb P\r)$. 
As $f(a^{-1}\langle {\bf x},a \theta_*\rangle, \delta)=f(\langle {\bf x},\theta_*\rangle, \delta)$ for any $a>0$, the best one can hope for is to recover 
$\theta_*$ up to a scaling factor. 
Hence, without loss of generality, we will assume that $\theta_\ast$ satisfies 
$\|\mf\Sigma^{1/2}\theta_\ast\|^2_2:=\dotp{\mf\Sigma^{1/2}\theta_\ast}{\mf\Sigma^{1/2}\theta_\ast}=1$, where $\mf\Sigma=\mb E(\mf x-\mb E\mf x)(\mf x-\mb E\mf x)^T$ is the covariance matrix of $\mf x$.
%Hence, without loss of generality, we will assume that $\theta_\ast$ belongs to the unit ball in the Euclidean norm, i.e. 
%$\l\|\theta_\ast\r\|_2:=\sqrt{\dotp{\theta_\ast}{\theta_\ast}}=1$.

In many applications, $\theta_\ast$ possesses special structure, such as sparsity or low rank (when $\theta_\ast\in \mb R^{d_1\times d_2}, \ d_1d_2=d$, is a matrix). 
To incorporate such structural assumptions into the problem, we will assume that $\theta_*$ is an element of a closed set $\Theta\subseteq\mathbb{R}^d$ of small ``statistical complexity'' that is characterized by its Gaussian mean width \citep{vershynin2015estimation}. %\cite{vershynin2015estimation}.  
The past decade has witnessed significant progress related to estimation in high-dimensional spaces, both in theory and applications. 
Notable examples include sparse linear regression \citep{tibshirani1996regression,sparse,bickel2009simultaneous}, low-rank matrix recovery (\cite{low-rank,gross2011recovering,chandrasekaran2012convex}), and mixed structure recovery \citep{mixed-recovery}. 
However, the majority of the aforementioned works assume that the link function $f$ is linear, and their results apply only to this particular case.  
%However, the problem as stated here is different for the following reasons:
%\begin{itemize}
%\item  Techniques such as the LASSO, or nuclear norm minimization, are built upon the linear data model $y_i=\dotp{\mathbf{x}_i}{\theta}+\delta_i$, \lcolor{for which the link function $f(\cdot)$ is assumed known.}
%\item The theoretical recovery guarantee of these methods requires strong moment assumptions, such as \lcolor{that the distributions of both the measurement vectors $\mathbf{x}_i, i=1,\ldots,n$, and the noise $\delta_i,i=1,\ldots,n$, are  
%sub-Gaussian.}
%\end{itemize}

Generally, the task of estimating the index vector requires approximating the link function $f$ \citep{hardle1993optimal} or its derivative, assuming that it exists (the so-called Average Derivative Method), see \citep{stoker1986consistent,hristache2001direct}. 
However, when the measurement vector $\mf{x}$ is Gaussian, a somewhat surprising result states that one can estimate $\theta_\ast$ directly, avoiding preliminary link function estimation step completely.  
More specifically, \cite{Brillinger1983A-generalized-l00} proved that 
$\eta\theta_\ast = \argmin_{\theta\in\mb R^d}\mb E\l( y - \dotp{\theta}{\mf{x}}\r)^2$, where $\eta = \mb E\dotp{y\mf{x}}{\theta_\ast}$. 
Later, \cite{li1989regression} extended this result to the more general case of elliptically symmetric distributions, which includes the Gaussian as a special case; see Lemma \ref{lemma:mean-consistency}. 
In general, it is not always possible to recover $\theta_\ast$: see \citep{ai2014one} for an example in the case when $f(x)=sign(x)$ (so-called ``1-bit compressed sensing'' \citep{boufounos20081}). 

%However, when the measurement vectors $\mf{x_j}, \ j=1,\ldots,n$ are Gaussian, direct construction of an unbiased estimator of $\theta_\ast$ that does not involve estimation of the link function is known and dates back to the work of \cite{Brillinger1983A-generalized-l00}. 
Y. Plan, R. Vershynin and E. Yudovina recently presented the non-asymptotic study for the case of Gaussian measurements in the context of high-dimensional structured estimation \citep{plan2014high,plan2016generalized}; also, see \citet{genzel2016high,ai2014one,thrampoulidis2015lasso,yi2015optimal} for further details. 
On a high level, these works show that when $\mf{x_j}$'s are Gaussian, nonlinearity can be treated as an additional noise term. 
To give an example, \cite{plan2016generalized} and \cite{plan2014high} demonstrate that under the same model as \eqref{model}, when
%\lcomm{need a bit more detail on what the model is here, in particular how similar it is to the model considered here} 
$\mathbf{x}_j\sim\mathcal{N}(0,\mathbf{I}_{d\times d})$, $\theta_*\in\Theta$, and $y_j$ is sub-Gaussian for $j=1,\ldots,n$, solving the constrained problem
\[
\widehat{\theta}=\argmin_{\mathbf{\theta}\in \Theta}~\|\mathbf{y}-\mathbf{X}\theta\|_2^2,
\]
with $\mathbf{y}=[y_1~\cdots~y_m]^T$ and $\mathbf{X}=\frac{1}{\sqrt{m}}[\mathbf{x}_1~\cdots~\mathbf{x}_m]^T$, recovers $\theta_*$ up to a scaling factor $\eta$ with high probability: namely, for all $\beta \ge 2$,
\begin{align}
\label{bound-1}
&
\mb P\left[\left\|\widehat{\theta}-\eta\theta_*\right\|_2\geq 
C\frac{\omega(D(\Theta,\eta\theta_*)\cap\mathbb{S}^{d-1})+\beta}{\sqrt{m}}\right]\leq ce^{-\beta^2/2},
\end{align}
where, with formal definitions to follow in Section \ref{sec:back}, $\mathbb{S}^{d-1}$ is the unit sphere in $\mathbb{R}^d$, $D(\Theta,\theta)$ is the descent cone of $\Theta$ at point $\theta$ and $\omega(T)$ is the Gaussian mean width of a subset $T \subset \mathbb{R}^d$.
%\lcomm{keeping definitions general, and easier to read}  
A different approach to estimation of the index vector in model \eqref{model} with similar recovery guarantees has been developed in \cite{yi2015optimal}. 
However, the key assumption adopted in all these works that the vectors $\mf{x_j}$ follow Gaussian distributions preclude situations where the measurements are heavy tailed, and hence might be overly restrictive for some practical applications; for example, noise and outliers observed in high-dimensional image recovery often exhibit heavy-tailed behavior, see \cite{face-recognition}.

As we mentioned above, \cite{li1989regression} have shown that direct consistent estimation of $\theta_\ast$ is possible when $\Pi$ belongs to a family of elliptically symmetric distributions.
%In particular, the class of elliptically symmetric distributions includes Gaussian distribution as a special case. 
Our main contribution is the non-asymptotic analysis for this scenario, with a particular focus on the case when $d>n$ and $\theta_\ast$ possesses special structure, such as sparsity. 
Moreover, we make very mild assumptions on the tails of the response variable $y$: for example, when the link function satisfies 
$f(\dotp{\mf{x}}{\theta_\ast},\delta)=\tilde f(\dotp{\mf{x}}{\theta_\ast})+\delta$, it is only assumed that $\delta$ possesses $2+\eps$ moments, for some $\eps>0$. 
\cite{plan2016generalized} present analysis for the Gaussian case and ask ``Can the same kind of accuracy be expected for random non-Gaussian matrices?'' In this paper, we give a positive answer to their question. 
To achieve our goal, we propose a Lasso-type estimator that admits tight probabilistic guarantees in spirit of \eqref{bound-1} despite weak tail assumptions (see Theorem \ref{master-bound} below for details).

Proofs of related non-asymptotic results in the literature rely on special properties of Gaussian measures. 
To handle a wider class of elliptically symmetric distributions, we rely on recent developments in generic chaining methods \citep{Talagrand-book-2,Mendelson-2}. 
These general tools could prove useful in developing further extensions to a wider class of design distributions. 
%obtained by optimizing a \lcolor{a certain quadratic loss function in which a heavy tailed term has been truncated.}
% We show that under the appropriate truncation level, the proposed  \lcolor{estimator is robust} and achieves the same error rate as \eqref{bound-1}, for general heavy-tailed elliptical \lcolor{measurement vectors} $\mathbf{x}_i$'s, \lcolor{and under only weak moment assumptions on the observed $y_i$'s.}

%###############
\section{Definitions and background material.} 
\label{sec:back}
%###############

This section introduces main notation and the key facts related to elliptically symmetric distributions, convex geometry and empirical processes. 
The results of this section will be used repeatedly throughout the paper.
\\
For the unified treatment of vectors and matrices, it will be convenient to treat a vector $v\in \mb R^{d\times 1}$ as a $d\times 1$ matrix. 
Let $d_1,d_2\in \mb N$ be such that $d_1 d_2=d$. 
Given $v_1,v_2\in \mb R^{d_1\times d_2}$, the Euclidean dot product is then defined as $\dotp{v_1}{v_2}=\tr(v_1^T v_2)$, where $\tr(\cdot)$ stands for the trace of a matrix and $v^T$ denotes the transpose of $v$. \\
The $\ell_1$-norm of $v\in \mb R^d$ is defined as $\|v\|_1=\sum_{j=1}^d |v_j|$. 
The nuclear norm of a matrix $v\in \mb R^{d_1\times d_2}$ is  
$\|v\|_\ast = \sum_{j=1}^{\min(d_1,d_2)} \sigma_j(v)$, where $\sigma_j(v), \ j=1,\ldots,\min(d_1,d_2)$ stand for the singular values of $v$, and 
the operator norm is defined as $\|v\|=\max_{j=1,\ldots,\min(d_1,d_2)} \sigma_j(v)$.  

%####################################
\subsection{Elliptically symmetric distributions.} 
%####################################

%We start \lcolor{define a family of elliptical distributions and discuss some of} its important properties. 
A centered random vector $\mathbf{x}\in\mathbb{R}^d$ has elliptically symmetric (alternatively, elliptically contoured or just elliptical) distribution with parameters $\mathbf{\Sigma}$ and $F_{\mu}$, denoted $\mathbf{x}\sim\mathcal{E}(0,~\mathbf{\Sigma},~F_{\mu})$, 
if 
\begin{equation}
\label{elliptical-definition}
\mathbf{x}\stackrel{d}{=}\mu\mathbf{B}U,
\end{equation}
where $\stackrel{d}{=}$ denotes equality in distribution, $\mu$ is a scalar random variable with cumulative distribution function $F_{\mu}$, $\mathbf{B}$ is a fixed $d\times d$ matrix such that $\mathbf{\Sigma}=\mathbf{B}\mathbf{B}^T$, and $U$ is uniformly distributed over the unit sphere $\mathbb{S}^{d-1}$ and independent of $\mu$. 
Note that distribution $\mathcal{E}(0,~\mathbf{\Sigma},~F_{\mu})$ is well defined, as if $\mathbf{B}_1\mathbf{B}_1^T=\mathbf{B}_2\mathbf{B}_2^T$, then there exists a unitary matrix $\mathbf{Q}$ such that $\mathbf{B}_1=\mathbf{B}_2\mathbf{Q}$, and $\mathbf{Q}U\stackrel{d}{=}U$. Along these same lines, we note that representation \eqref{elliptical-definition} is not unique, as one may replace the pair $(\mu,~\mathbf{B})$ with $\left(c\mu,~\frac{1}{c}\mathbf{B}\mathbf{Q}\right)$ for any constant $c>0$ and any orthogonal matrix $\mathbf{Q}$. 
To avoid such ambiguity, in the following we allow $\mathbf{B}$ to be any matrix satisfying $\mathbf{B}\mathbf{B}^T=\mathbf{\Sigma}$, and noting that the covariance matrix of $U$ is a multiple of the identity, we further impose the condition that the covariance matrix of $\mathbf{x}$ is equal to $\mathbf{\Sigma}$, i.e. $\expect{\mathbf{x}\mathbf{x}^T}=\mathbf{\Sigma}$.

Alternatively, the mean-zero elliptically symmetric distribution can be defined uniquely via its characteristic function  
\[
\mathbf{s}\rightarrow\psi\left(\mathbf{s}^T\mathbf{\Sigma}\mathbf{s}\right),~\mathbf{s}\in\mathbb{R}^d,
\]
where $\psi:\mathbb{R}^+\rightarrow\mathbb{R}$ is called the characteristic generator of $\mathbf{x}$. 
For further details information about elliptically  distribution, see \citep{elliptical-paper} for details.

An important special case of the family $\mathcal{E}(0,~\mathbf{\Sigma},~F_{\mu})$
of elliptical distributions is the Gaussian distribution $\mathcal{N}(0,\mathbf{\Sigma})$, where $\mu=\sqrt{z}$ with $z \stackrel{d}{=} \chi_d^2$, and the characteristic generator is $\psi(x)=e^{-x/2}$.

The following elliptical symmetry property, generalizing the well known fact for the conditional distribution of the multivariate Gaussian, plays an important role in our subsequent analysis, see \citep{elliptical-paper}:
\begin{proposition}
\label{elliptical-theorem}
Let $\mathbf{x}=[\mathbf{x}_1,~\mathbf{x}_2]\sim\mathcal{E}_d(0,\mathbf{\Sigma},F_\mu)$, where are of dimension $d_1$ and $d_2$ respectively, with $d_1+d_2=d$. 
Let $\mf{\Sigma}$ be partitioning accordingly as
\[
\mathbf{\Sigma}=
\left[
\begin{array}{cc}
\mathbf{\Sigma}_{11}  & \mathbf{\Sigma}_{12}  \\
\mathbf{\Sigma}_{21} & \mathbf{\Sigma}_{22}  
\end{array}
\right].
\]
Then, whenever $\mathbf{\Sigma}_{22}$ has full rank,
%\lcomm{like the Gaussian case, should only need $\Sigma_{22}$ to be full rank.} 
the conditional distribution of ${\bf x}_1$ given ${\bf x}_2$ is elliptical 
$\mathcal{E}_{d_1}(0,\mathbf{\Sigma}_{1|2},F_{\mu_{1|2}})$, where
\[
\mathbf{\Sigma}_{1|2}=\mathbf{\Sigma}_{11}-\mathbf{\Sigma}_{12}\mathbf{\Sigma}_{22}^{-1}\mathbf{\Sigma}_{21},
\]
and $F_{\mu_{1|2}}$ is the cumulative distribution function of $(\mu^2-\mathbf{x}_2^T\mathbf{\Sigma}_{22}^{-1}\mathbf{x}_2)^{1/2}$ given 
$\mathbf{x}_2$.
\end{proposition}
Note that $\mu^2-\mathbf{x}_2^T\mathbf{\Sigma}_{22}^{-1}\mathbf{x}_2$ is always nonnegative, hence $F_{\mu_{1|2}}$ is well defined, since by \eqref{elliptical-definition} we have
\begin{align*}
\mathbf{x}_2^T\mathbf{\Sigma}_{22}^{-1}\mathbf{x}_2
=\mu^2(\mathbf{B}_2 U)^T(\mathbf{B}_2\mathbf{B}_2^T)^{-1}(\mathbf{B}_2 U)
=\mu^2 U^T\mathbf{B}_2^T(\mathbf{B}_2\mathbf{B}_2^T)^{-1}\mathbf{B}_2 U
\leq\mu^2 U^T U=\mu^2,
\end{align*}
where $\mathbf{B}_2$ is the matrix consisting of the last $d_2$ rows of $\mathbf{B}$ in \eqref{elliptical-definition}, and where the inequality holds due to the fact that $\mathbf{B}_2^T(\mathbf{B}_2\mathbf{B}_2^T)^{-1}\mathbf{B}_2$ is a  projection matrix. 
The following corollary is easily deduced from the theorem above:
\begin{corollary}
\label{elliptical-corollary}
If $\mathbf{x}\sim\mathcal{E}_d(0,\mathbf{\Sigma},F_\mu)$ with $\mathbf{\Sigma}$ of full rank, then for any two fixed vectors $\mathbf{y}_1,\mathbf{y}_2\in\mathbb{R}^d$ with $\|\mathbf{y}_2\|_2=1$,
\[\expect{\langle\mathbf{x},\mathbf{y}_1\rangle~|~\langle\mathbf{x},\mathbf{y}_2\rangle}
=\langle\mathbf{y}_1,\mathbf{y}_2\rangle\langle\mathbf{x},\mathbf{y}_2\rangle.\]
\end{corollary}
\begin{proof}
Let $\{\mathbf{v}_1,\cdots,\mathbf{v}_d\}$ be an orthonormal basis in $\mathbb{R}^d$ such that $\mathbf{v}_d=\mathbf{y}_2$. 
Let $\mathbf{V}=[\mathbf{v}_1
~\mathbf{v}_2~\cdots~\mathbf{v}_d]$ and consider the linear transformation
%\lcomm{made widetilde}
\[
\widetilde{\mathbf{x}}=\mathbf{V}^T\mathbf{x}.
\]
Then, by \eqref{elliptical-definition}, $\widetilde{\mathbf{x}}=\mu\mathbf{V}^T\mathbf{B} U$, which is centered elliptical with full rank covariance matrix 
$\mathbf{V}^T\mathbf{\Sigma}\mathbf{V}$. 
Applications of Theorem \ref{elliptical-theorem} with $\mathbf{x}_1=[\langle\mathbf{x},\mathbf{v}_1\rangle,~\cdots,~\langle\mathbf{x},\mathbf{v}_{d-1}\rangle]$ and $\mathbf{x}_2=\langle\mathbf{x},\mathbf{v}_{d}\rangle=\langle\mathbf{x},\mathbf{y}_2\rangle$ yields
\begin{align*}
\expect{\langle\mathbf{x},\mathbf{y}_1\rangle~|~\langle\mathbf{x},\mathbf{y}_2\rangle}
=&\expect{\left.\sum_{i=1}^d\langle\mathbf{x},\mathbf{v}_i\rangle\langle\mathbf{y}_1,\mathbf{v}_i\rangle~\right|~\langle\mathbf{x},\mathbf{v}_d\rangle}\\
=&\expect{\left.\sum_{i=1}^{d-1}\langle\mathbf{x},\mathbf{v}_i\rangle\langle\mathbf{y}_1,\mathbf{v}_i\rangle~\right|~\langle\mathbf{x},\mathbf{v}_d\rangle}
+\langle\mathbf{x},\mathbf{v}_d\rangle\langle\mathbf{y}_1,\mathbf{v}_d\rangle\\
=&\langle\mathbf{x},\mathbf{v}_d\rangle\langle\mathbf{y}_1,\mathbf{v}_d\rangle
=\langle\mathbf{y}_1,\mathbf{y}_2\rangle\langle\mathbf{x},\mathbf{y}_2\rangle,
\end{align*}
where in the second to last equality we have used the fact that the conditional distribution of  $[\langle\mathbf{v}_1,\mathbf{x}\rangle,~\cdots,~\langle\mathbf{v}_{d-1},\mathbf{x}\rangle]$ given $\langle\mathbf{x},\mathbf{v}_d\rangle$ is 
elliptical with mean zero.
\end{proof}

\subsection{Geometry.}
\begin{definition}[Gaussian mean width]
\label{gmw}
The Gaussian mean width of a set $T \subseteq \mathbb{R}^d$ is defined as
\[
\omega(T):=\expect{\sup_{t\in T}~\langle\mathbf{g},t\rangle},
\]
where $\mathbf{g}\sim\mathcal{N}(0,\mathbf{I}_{d\times d})$. 
%\lcomm{think we don't need: i.e. it is distributed according to the standard normal.}
\end{definition}

\begin{definition}[Descent cone]
\label{DC}
The descent cone of a set $\Theta\subseteq\mathbb{R}^d$ at a point $\theta\in\mathbb{R}^d$ is defined as
\[
D(\Theta,\theta)=\{\tau\mathbf{h}:~\tau\geq0, \mathbf{h}\in \Theta-\theta\}.
\]
\end{definition}

\begin{definition}[Restricted set]
\label{def:restricted.set}
Given $c_0>1$, the $c_0$-restricted set of the norm $\|\cdot\|_{\mathcal{K}}$ at $\theta\in\mathbb{R}^d$ is defined as
\begin{align}
\label{eq:rest-set}
&
\mathbb{S}_{c_0}(\theta):=\mathbb{S}_{c_0}(\theta;\m K)
=\left\{\mathbf{v}\in\mathbb{R}^d:~\|\theta+\mathbf{v}\|_{\mathcal{K}}
\leq \|\theta\|_{\mathcal{K}}+\frac{1}{c_0}\|\mathbf{v}\|_{\mathcal{K}}\right\}.
\end{align}
%\[\mb{S}_{c_0}:=\left\{\mathbf{x}\in\mathbb{R}^d:~\|P_{\mathcal{S}^{\perp}}\mathbf{x}\|_{\mathcal{K}}\leq c_0\|P_{\mathcal{S}}\mathbf{x}\|_{\mathcal{K}}\right\},\]
%for some constant $c_0>1$.
\end{definition}
\begin{definition}[Restricted compatibility]\label{def:restricted.compatibility}
The restricted compatibility constant of a set $A \subseteq \mathbb{R}^d$ with respect to the norm $\|\cdot\|_{\mathcal{K}}$ is given by%\lcomm{general definition is less cluttered and easier to parse}
\[
\Psi(A):=\Psi(A;\m K) =
\sup_{\mathbf{v}\in A\backslash \{0\}}\frac{\|\mathbf{v}\|_{\mathcal{K}}}{\|\mathbf{v}\|_2}.
\]
\end{definition}
\begin{remark}
The restricted set from the definition \ref{def:restricted.set} is not necessarily convex. 
%and its restricted compatibility constant is not easy to obtain. 
However, if the norm $\|\cdot\|_{\mathcal{K}}$ is decomposable (see definition \ref{def:decomposable}), then the restricted set is contained in a convex cone, and the corresponding restricted compatibility constant is easier to estimate. 
Decomposable norms have been introduced by \citet{general-m-estimator} and later appeared in a number of works, e.g. \citep{m-estimator-2} and references therein. 
For reader's convenience, we provide a self-contained discussion in Appendix \ref{app-B}. 
\end{remark}

%The proof can be found in \cite{introduction-to-random-matrix}.
%#################################
\section{Main results.}
%#################################

In this section, we define a version of Lasso estimator that is well-suited for heavy-tailed measurements, and state its performance guarantees. 

%As we assume the measurement vectors are mean zero and elliptically symmetric with $\mf{\Sigma}$ known, by replacing $\mathbf{x}$ by 
%$\mathbf{\Sigma}^{-1/2} \mathbf{x}$, 
%without loss of generality 

We will assume that $\mathbf{x}_1,~\mathbf{x}_2,~\ldots,~\mathbf{x}_m\in \mb R^d$ are i.i.d. copies of an \textbf{isotropic} vector $\mathbf{x}$ with spherically symmetric distribution $\mathcal{E}_d(0,\mathbf{I_{d\times d}},F_{\mu})$. 
If $\mf{x} \sim \mathcal{E}_d(0,\mf{\Sigma},F_{\mu})$ for some positive definite matrix $\mf\Sigma$, then by definition $\mf{x}\stackrel{d}{=}\mu\mathbf{\Sigma}^{1/2}U$, and 
$\dotp{\mf{x}}{\theta_\ast}=\dotp{\mf{\Sigma}^{-1/2}\mf{x}}{\mf\Sigma^{1/2}\theta_\ast}$, where $\mf{\Sigma^{-1/2}\mf{x}} = \mu U \sim \mathcal{E}_d(0,\mathbf{I_{d\times d}},F_{\mu})$. 
Hence, if we set $\tilde\theta_\ast:=\mf\Sigma^{1/2}\theta_\ast$, then all results that we establish for isotropic measurements hold with 
$\theta_\ast$ replaced by $\tilde\theta_\ast$; remark after Theorem \ref{master-bound} includes more details. 
%\lcolor{Also letting the noise variables be i.i.d. copies of a variable $\delta$,} the corresponding observations $y_1,~y_2,~\cdots,~y_m$ \lcolor{are as specified in} model \eqref{model}, \lcolor{and thus} i.i.d. copies of 
%$y:=f(\dotp{\mathbf{x}}{\theta_*},\delta)$.
%$X_i$ $$\mathbf{x}_i$ and $y_i$ and present the corresponding performance guarantee.

\subsection{Description of the proposed estimator.}
%\marginpar{\Stas{It probably makes more sense to assume that measurements are spherically symmetric right away}}
%Suppose $X_1,\cdots,X_1$ are i.i.d. samples from a mean-zero elliptical distribution $\mathcal{E}_d(0,\mathbf{\Sigma},F_{\mu})$ with full rank covariance matrix $\mathbf{\Sigma}$ known to us. We propose to take the measurement vector 
%$\mathbf{x}_i=\mathbf{\Sigma}^{-1/2}\mathbf{X}_i$, which guarantees that $\expect{\mathbf{x}_i\mathbf{x}_i^T}=\mathbf{I}_{d\times d}$. 

We first introduce an estimator under the scenario that $\theta_*\in\Theta$, for some known closed set $\Theta\subseteq\mathbb{R}^d$. Define the loss function $L_m^0(\cdot)$ as 
%\[
%\mathbf{X}:=[x_1~\cdots~x_n]^T,
%\]
\begin{align}
\label{eq:loss0}
&
L^0_{m}(\theta):=\|\theta\|_2^2  - \frac{2}{m}\sum_{i=1}^m \dotp{y_i \mathbf{x}_i }{\theta},
\end{align}
which is the unbiased estimator of
\[
L^0(\theta) := \|\theta\|_2^2 - 2\mb E\dotp{y\mathbf{x}}{\theta} = \mb E\l( y - \dotp{x}{\theta}\r)^2 - \mb E y^2,
\]
where the last equality follows since $x$ is isotropic. 
Clearly, minimizing $L^0(\theta)$ over any set $\Theta\subseteq \mb R^d$ is equivalent to minimizing the quadratic loss
%\lcomm{its not so clear what the expectation is being taken over here, as the distribution of $y$ depends on $\theta$, but that's not what is intended here}\xcomm{$y$ depends on $\theta^*$, not $\theta$, thus the expectation should be well-defined.}
$\mb E\l(  y - \dotp{\mathbf{x}}{\theta} \r)^2$.
% since $\expect{\dotp{x}{\theta}^2}=\|\theta\|_2^2$. 
If distribution $F_\mu$ has heavy tails, the sample average 
$\frac{1}{m}\sum_{i=1}^m y_i \mathbf{x}_i$ might not concentrate sufficiently well around its mean, hence we replace it by a more ``robust'' version obtained via truncation. 
%and decompose the quadratic loss function $\frac1m\|\mathbf{Y}-\mathbf{X}\mathbf{\theta}\|_2^2$ as follows:
%\begin{align*}
%\frac1m\|\mathbf{y}-\mathbf{A}\mathbf{x}\|_2^2
%=&\frac1m\|\mathbf{y}\|_2^2-\frac2m\langle\mathbf{Ax},\mathbf{y}\rangle+\frac1m\|\mathbf{Ax}\|_2^2\\
%=&-\frac{2}{m}\sum_{i=1}^m\langle\mathbf{x}_i,\mathbf{x}\rangle y_i+\frac1m\|\mathbf{Ax}\|_2^2+\frac1m\|\mathbf{y}\|_2^2.
%\end{align*}
Let $\mu\in \mb R$, $U\in \mb S^{d-1}$ be such that $\mathbf{x} = \mu U$ (so that $\mu=\|\mf{x}\|_2)$, and set
\begin{align}
\label{transformation}
&
\widetilde{U}=\sqrt{d}U,\\
&\nonumber
q=\mu y/\sqrt{d}, 
\end{align}
so that $q \widetilde U = y \mathbf{x}$ and $\widetilde U$ is uniformly distributed on the sphere of radius $\sqrt d$, implying that its covariance matrix is $I_d$, the identity matrix. 
Next, define the truncated random variables
\begin{align}
\label{transformation-2}
\widetilde{q}_i=\sign{(q_i)}(|q_i|\wedge\tau), \ i=1,\ldots, m,
\end{align}
where $\tau=m^{\frac{1}{2(1+\kappa)}}$ for some $\kappa\in(0,1)$ that is chosen based on the integrability properties of $q$, see \eqref{eq:determine.kappa}.
%$$\langle\mathbf{x}_i,\mathbf{x}\rangle y_i=\langle\widetilde{U}_i,\mathbf{x}\rangle q_i$$ 
%and
%minimizing $\|\mathbf{y}-\mathbf{A}\mathbf{x}\|_2^2$ is equivalent to minimizing 
%\begin{equation}\label{original-form}
%-\frac{2}{m}\sum_{i=1}^m\langle\widetilde{U}_i,\mathbf{x}\rangle q_i+\frac1m\|\mathbf{Ax}\|_2^2.
%\end{equation}
%Now, we are ready to introduce our robust estimator. First of all, by \eqref{transformation},
%$\expect{\widetilde{\mathbf a}_i\widetilde{\mathbf a}_i^T}=\mathbf{I}_{d\times d}$. Thus, we make use of this information and replace $\|\mathbf{Ax}\|_2^2$ in \eqref{original-form} by its mean, which is just $m\|\mathbf{x}\|_2^2$. Then, define the following truncated version of $q_i$:
%where $a\wedge b:=\min\{a,b\}$ and $\tau$ is the truncation level. Here we take $\tau=\mathcal{O}\left(m^{1/4}\right)$. Next, we estimate $\theta_*$ by solving the following robust LASSO:
Finally, set 
\begin{align}
\label{eq:loss}
L^\tau_{m}(\theta)=\|\theta\|_2^2  - \frac{2}{m}\sum_{i=1}^m\dotp{\widetilde q_i\widetilde U_i}{\theta},
\end{align}
and define the estimator $\widehat \theta_m$ as the solution to the constrained optimization problem:
%\lcomm{need to justify argmin lies in $\Theta$, as used later in claims about descent cone, e.g. see line above \eqref{eq:main}}
\begin{align}
\label{eq:truncated-version}
\widehat\theta_m:=\argmin\limits_{\theta\in \Theta}L^\tau_m(\theta).
\end{align}
%\xcolor{Since $\Theta$ is closed, $\widehat\theta_m$ must lie inside $\Theta$.}
We will also denote  
\begin{align}
\label{eq:exp-loss}
L^\tau(\theta) := \mb E L^\tau_m(\theta) = \|\theta\|_2^2 - 2\mb E \dotp{\widetilde q\widetilde U}{\theta}. 
\end{align}
For the scenarios where structure on the unknown $\theta_\ast$ is induced by a norm $\|\cdot\|_{\mathcal{K}}$ (e.g., if $\theta_\ast$ is sparse, then $\|\cdot\|_\m K$ could be the $\|\cdot\|_1$ norm), we will also consider the estimator $\widehat{\theta}^{\lambda}_{m}$ defined via 
\begin{equation}
\label{eq:unconstrained-version}
\widehat{\theta}^{\lambda}_m:=\argmin_{\theta\in\mathbb{R}^d}\Big[ L_m^{\tau}(\theta)+\lambda\|\theta\|_{\mathcal{K}} \Big],
\end{equation}
where $\lambda>0$ is a regularization parameter to be specified, and $L_m^{\tau}(\theta)$ is defined in \eqref{eq:loss}.

Let us note that truncation approach has previously been successfully implemented by \cite{truncation-paper} to handle heavy-tailed noise in the context of matrix recovery with sub-Gaussian design.
In the present paper, we show that truncation-based approach is also useful in the situations where the measurements are heavy-tailed.

\begin{remark}
Note that our estimator \eqref{eq:unconstrained-version} is in general much easier to implement than some other popular alternatives, such as the usual Lasso estimator \citep{tibshirani1996regression}. 
For example, when the signal $\theta$ is sparse, our estimator takes the form
\[
\widehat{\theta}^{\lambda}_m:=\argmin_{\theta\in\mathbb{R}^d}\Big[ \|\theta\|_2^2  - \frac{2}{m}\sum_{i=1}^m\dotp{\widetilde q_i\widetilde U_i}{\theta} + \lambda\|\theta\|_{1} \Big],
\]
which yields a closed form solution in the form of ``soft-thresholding''. 
Specifically, let 
$\mathbf{b}=\frac{1}{m}\sum_{i=1}^m\widetilde q_i\widetilde U_i$, then, the $k$-th entry of 
$\widehat{\theta}^{\lambda}_m$ takes the form:
\begin{align}\label{ST-estimator}
\left(\widehat{\theta}^{\lambda}_m\right)_k
=\begin{cases}
b_k-\lambda/2,~~&\textrm{if}~~b_k\geq\lambda/2,\\
0,~~&\textrm{if}~~   -\lambda/2\leq b_k\leq\lambda/2,\\
b_k+\lambda/2,~~&\textrm{if}~~b_k\leq   -\lambda/2.
\end{cases}
\end{align}
We should note however that such simplification comes at the cost of knowing the distribution of measurement vector $\mf{x}$. 
Despite being of low computational complexity, our estimator can still exploit the structure of the problem, while being robust both to the possible model misspecification as well as to data corruption modeled by the heavy-tailed distributions. 
We demonstrate this in the following sections. 
\end{remark}
\begin{remark}[Non-isotropic measurements]
\label{rmk:non-isotropic}
When $\mf x\sim  \mathcal{E}_d(0,\mf{\Sigma},F_{\mu})$ for some $\Sigma\succ 0$, then estimator \eqref{eq:truncated-version} has to be replaced by 
\begin{align}
\label{non-isotropic1}
\widehat\theta_m:=\argmin\limits_{\theta\in \Theta} 
\Big[
\| \mf\Sigma^{1/2}\theta\|_2^2  - \frac{2}{m}\sum_{i=1}^m\dotp{\widetilde q_i\widetilde U_i}{ \mf\Sigma^{1/2}\theta}
\Big],
\end{align}
which is equivalent to 
\[
\tilde\theta_m:=\argmin\limits_{\theta\in \mf\Sigma^{1/2}\Theta} 
\Big[
\|\theta\|_2^2  - \frac{2}{m}\sum_{i=1}^m\dotp{\widetilde q_i\widetilde U_i}{\theta}
\Big],
\]
is a sense that $\tilde\theta_m = \mf\Sigma^{1/2}\hat\theta_m$. 
Hence, results obtained for isotropic measurements easily extend to the more general case. 
Similarly, estimator \eqref{eq:unconstrained-version} should be replaced by
\begin{align}
\label{non-isotropic2}
\hat{\theta}^{\lambda}_m:=\argmin_{\theta\in\mathbb{R}^d}\Big[ 
\| \mf\Sigma^{1/2}\theta\|_2^2  - \frac{2}{m}\sum_{i=1}^m\dotp{\widetilde q_i\widetilde U_i}{\Sigma^{1/2}\theta} + \lambda\| \mf\Sigma^{1/2}\theta\|_{\mathcal{K}} 
\Big],
\end{align}
which is equivalent to 
\[
\tilde{\theta}^{\lambda}_m:=\argmin_{\theta\in\mathbb{R}^d}\Big[ 
\|\theta\|_2^2  - \frac{2}{m}\sum_{i=1}^m\dotp{\widetilde q_i\widetilde U_i}{\theta} + \lambda\|\theta\|_{\mf\Sigma^{1/2}\mathcal{K}} 
\Big],
\]
meaning that $\tilde{\theta}^{\lambda}_m = \mf\Sigma^{1/2}\hat{\theta}^{\lambda}_m$. 
\end{remark}

%\lcomm{I think the material now inserted at the start of the section is sufficient, so remark that was here is now commented out.}
%\begin{remark}
%Suppose a vector $\mathbf{x}$ is drawn from a general %non-isotropic elliptical distribution %$\mathcal{E}_d(0,\mathbf{\Sigma},F_{\mu})$. Then, as long %as $\mathbf{\Sigma}$ is full rank and known to us, we can %always take the measurement vector as %$\mathbf{\Sigma}^{-1/2}\mathbf{x}$ which is isotropic. %Thus, by assuming that %$\mathbf{\Sigma}=\mathbf{I}_{d\times d}$ right from the %beginning, there is no loss of generality.
%\end{remark}

%################################
\subsection{Estimator performance guarantees.}
%################################

In this section, we present the probabilistic guarantees for the performance of the estimators $\widehat \theta_m$ and $\widehat \theta^\lambda_m$ defined by \eqref{eq:truncated-version} and \eqref{eq:unconstrained-version} respectively. \\
Everywhere below, $C,c,C_j$ denote numerical constants; when these constants depend on parameters of the problem, we specify this dependency by writing $C_j=C_j(\text{parameters})$. 
%We will first present Our main result for the estimator  is the following error bound similar to \eqref{bound-1} which was shown for the Gaussian case:
Let 
\begin{align}
\label{eq:eta}
&
\eta=\mb E\dotp{y\mf{x}}{\theta_\ast},
\end{align} 
and assume that $\eta\ne0$ and $\eta\theta_\ast\in \Theta$. 
\begin{theorem}
\label{master-bound}
%Let the unknown $\theta_*\in\mathbb{S}^{d-1}$, and assume $\eta\theta_*\in \Theta\subseteq\mathbb{R}^d$ for $\eta=\expect{\dotp{ y\mathbf{x}}{\theta_*}}$, where 
Suppose that $\mathbf{x}\sim\mathcal{E}(0,~\mathbf{I}_{d\times d},~F_{\mu})$. 
Moreover, suppose that for some $\kappa>0$ 
%and a constant $\phi>0$ such that
\bea 
\label{eq:determine.kappa}
\phi:=\mb E|q|^{2(1+\kappa)}<\infty.
\ena
%where $q=\mu y/\sqrt{d}$.
%\lcomm{$\kappa$ is not arbitrary as given the distribution of $q$ there are some $\kappa$ for which the condition is satisfied and some for which not.}
Then there exist constants $C_1=C_1(\kappa,\phi),C_2=C_2(\kappa,\phi)>0$ such that $\widehat{\theta}_m$ satisfies
\begin{align*}
\mb P\left(\left\|\widehat{\theta}_m-\eta\theta_*\right\|_2\geq C_1\frac{(\omega(D(\Theta,\eta\theta_*)\cap\mathbb{S}^{d-1})+1)\beta}{\sqrt{m}}\right)\leq C_2e^{-\beta/2},
\end{align*}
for any $\beta\geq8$ and $m\geq \beta^2\l( \omega(D(\Theta,\eta\theta_*)\cap\mathbb{S}^{d-1})+1 \r)^2$.
\end{theorem}
\begin{remark}
\begin{enumerate}
\item Unknown link function $f$ enters the bound only through the constant $\eta$ defined in \eqref{eq:eta}. 
\item 
Aside from independence, conditions on the noise $\delta$ are implicit and follow from assumptions on $y$. 
In the special case when the error is additive, that is, when $y=f(\dotp{\mathbf{x}}{\theta_*})+\delta$, the moment condition \eqref{eq:determine.kappa} becomes 
$\mb E \big|\|\mf{x}\|_2 f(\dotp{\mathbf{x}}{\theta_*})+\|\mf{x}\|_2\delta\big|^{2(1+\kappa)}<\infty$, 
for which it is sufficient to assume that  $\mb E\Big|\|\mf{x}\|_2 f(\dotp{\mathbf{x}}{\theta_*})\Big|^{2(1+\kappa)}<\infty$ and 
$\mb E\left|\|\mf{x}\|_2\delta\right|^{2(1+\kappa)}<\infty$. 
\item 
Theorem \ref{master-bound} is mainly useful when $\eta\theta_*$ lies on the boundary of the set $\Theta$. 
Otherwise, if $\eta\theta_*$ belongs to the relative interior of $\Theta$, the descent cone $D(\Theta,\eta\theta_*)$ is the affine hull of 
$\Theta$  (which will often be the whole space $\mathbb{R}^d$). 
Thus, in such cases the Gaussian mean width $\omega(D(\Theta,\eta\theta_*)\cap\mathbb{S}^{d-1})$ can be on the order of $\sqrt{d}$, which is prohibitively large when $d\gg m$. 
We refer the reader to \citep{plan2016generalized,plan2014high} for a discussion of related result and possible ways to tighten them. 
\end{enumerate}
\end{remark}

%\lcomm{clearly sufficient, easy to see also necessary?}\xcomm{seems not necessary:(}
\noindent 
Next, we present performance guarantees for the unconstrained estimator \eqref{eq:unconstrained-version}.
\begin{theorem}
\label{master-bound-2}
Assume that the norm $\|\cdot\|_{\mathcal{K}}$ dominates the 2-norm, i.e. $\|\mathbf{v}\|_{\mathcal{K}}\geq\|\mathbf{v}\|_{2},~\forall \mathbf{v}\in\mathbb{R}^d$. 
%Let $\theta_*\in\mathbb{S}^{d-1}$ and 
Let $\mathbf{x}\sim\mathcal{E}(0,~\mathbf{I}_{d\times d},~F_{\mu})$, and suppose that for some $\kappa>0$ 
%and a constant $\phi>0$ such that
\[
\phi:=\mb E |q|^{2(1+\kappa)}<\infty.
\]
Then there exist constants $C_3=C_3(\kappa,\phi),C_4=C_4(\kappa,\phi)>0$ such that for all 
$\lambda\geq \frac{C_3\beta}{\sqrt{m}}\l(1+\omega(\mathcal{G}) \r)$ 
\begin{align*}
\mb P\left(
\left\|\widehat{\theta}_m^{\lambda}-\eta\theta_*\right\|_2\geq \frac{3}{2}\lambda\cdot\Psi\left(\mb{S}_2\left(\eta\theta_*\right)\right)
\right)\leq C_4e^{-\beta/2},
\end{align*}
for any $\beta \geq 8$ and $m \geq (\omega(\mathcal{G})+1)^2\beta^2$, where 
$\mathcal{G}:=\{\mathbf{x}\in\mathbb{R}^d:~\|\mathbf{x}\|_{\mathcal{K}}\leq1\}$ is the unit ball of $\|\cdot\|_\m K$ norm,
and $\mb{S}_2(\cdot)$ and $\Psi(\cdot)$ are given in Definitions \ref{def:restricted.set} and \ref{def:restricted.compatibility} respectively. 
\end{theorem}
\noindent
\begin{remark}[Non-isotropic measurements]
It follows from remark \ref{rmk:non-isotropic} and \eqref{non-isotropic1} that, whenever $\mf x\sim  \mathcal{E}_d(0,\mf{\Sigma},F_{\mu})$, 
inequality of Theorem \ref{master-bound} has the form
\begin{align*}
\mb P\left(\left\| \mf\Sigma^{1/2}\l( \widehat{\theta}_m-\eta\theta_* \r) \right\|_2\geq C_1\frac{\l(\omega \l( \mf\Sigma^{1/2} D(\Theta,\eta\theta_*)\cap\mathbb{S}^{d-1}\r)+1 \r) \beta}{\sqrt{m}}\right)\leq C_2e^{-\beta/2},
\end{align*}
which can be further combined with the bound 
\[
\omega \l( \mf\Sigma^{1/2} D(\Theta,\eta\theta_*)\cap\mathbb{S}^{d-1}\r)\leq \| \mf\Sigma^{1/2}\| \cdot \|\mf\Sigma^{-1/2}\| \, \omega \l( D(\Theta,\eta\theta_*)\cap\mathbb{S}^{d-1} \r),
\]
that follows from remark 1.7 in \citep{plan2016generalized}. 
Similarly, the inequality of Theorem \ref{master-bound-2} holds with 
\[
\m G_{\mf\Sigma^{1/2}} := \{\mathbf{x}\in\mathbb{R}^d:~\|\mathbf{x}\|_{ \mf\Sigma^{1/2}\mathcal{K}}\leq1\},
\]
the unit ball of $\|\cdot\|_{\mf\Sigma^{1/2}\m K}$ norm, in place of $\m G$. 
Namely, for all $\lambda\geq \frac{C_3\beta}{\sqrt{m}} \l(1+\omega(\mathcal{G}_{\mf \Sigma^{1/2}}) \r)$, 
\[
\mb P\left(
\left\| \mf \Sigma^{1/2} \l( \widehat{\theta}_m^{\lambda}-\eta\theta_* \r) \right\|_2 
\geq 
\frac{3}{2}\lambda \cdot \Psi \l( \mb{S}_2 \l( \eta \mf \Sigma^{1/2} \theta_*\r); \mf \Sigma^{1/2}\m K \r)
\right)\leq C_4e^{-\beta/2}
\]
Note that $\omega \l(	\mathcal{G}_{\mf \Sigma^{1/2}}\r) \leq \| \mf \Sigma^{1/2} \| \, \omega(\mathcal{G})$. 
Moreover, we show in Appendix \ref{app-B} that for a class of decomposable norms (which includes $\|\cdot\|_1$ and nuclear norm), the upper bounds for $\Psi \l( \mb{S}_2 \l( \eta \mf \Sigma^{1/2} \theta_*\r); \mf \Sigma^{1/2}\m K \r)$ and 
$\Psi \l(\mb S_2(\eta\theta_\ast) \r)$ differ by the factor of $\l\| \mf \Sigma^{-1/2} \r\|$. 
%In particular, estimation rate is preserved as long as $\mf \Sigma$ is well-conditioned. 
\end{remark}
%We remark that our deviation bounds are weaker than bounds obtained for the Gaussian case by \cite{plan2016generalized}, see \eqref{bound-1}, since the confidence factor $\beta$ appears as a multiplicative, and not the additive, term. 
%While this is the price we pay for handling a much wider class of underlying distributions, it is not clear if the bound is optimal, or if it is merely and artifact of techniques we use. 

%but the bound obtained does not shrink to zero as $m\rightarrow\infty$.
%#################
\subsection{Examples.}
%#################

We discuss two popular scenarios: estimation of the sparse vector and estimation of the low-rank matrix. 
\\
%###############################
\textbf{Estimation of the sparse signal. } Assume that there exists $J\subseteq \l\{1,\ldots,d\r\}$ of cardinality $s\leq d$ such that $\theta_{\ast,j}=0$ for $j\notin J$. 
Let $\Theta = \l\{ \theta\in \mb R^d: \ \|\theta\|_1\leq \| \eta\theta_\ast\|_1 \r\}$, with $\eta$ defined in \eqref{eq:eta}. 
In this case, it is well-known that $\omega^2\l( D(\Theta,\eta\theta_\ast)\cap \mb S^{d-1} \r)\leq 2 s\log(d/s)+\frac{5}{4}s$, see proposition 3.10 in \citep{chandrasekaran2012convex}, hence Theorem \ref{master-bound} implies that, with high probability,  
\begin{align}
\label{eq:bound-m1}
&
\left\| \widehat{\theta}_m-\eta\theta_* \right\|_2\lesssim \sqrt{\frac{s\log(d/s)}{m}}
\end{align}
as long as $m\gtrsim s\log(d/s)$. 
\\
We compare this bound to result of Theorem \ref{master-bound-2} for constrained estimator. 
Let $\|\cdot\|_\m K$ be the $\ell_1$ norm. 
It is well-know that $\omega(\m G)=\mb E\max_{j=1,\ldots, d}|g_j|\leq \sqrt{2\log(2d)}$, where $\mf g\sim \m N(0,\mf I_{d\times d})$. 
Moreover, we show in Appendix \ref{app-B} that $\Psi\left(\mb{S}_2\left(\eta\theta_*\right) \right)\leq 4\sqrt{s}$. 
Hence, for $\lambda \simeq \sqrt{\frac{\log(2d)}{m}}$, Theorem \ref{master-bound-2} implies that 
\[
\left\|\widehat{\theta}_m^{\lambda}-\eta\theta_*\right\|_2\lesssim \sqrt{\frac{s\log(d)}{m}}
\] 
with high probability whenever $m\gtrsim \log(2d)$. 
This bound is only marginally weaker than \eqref{eq:bound-m1} due to the logarithmic factor, however, definition of 
$\widehat{\theta}_m^{\lambda}$ does not require the knowledge of $\l\|  \eta\theta_\ast \r\|_1$, as we have already mentioned before. 
\\
%###############################
\textbf{Estimation of a low-rank matrix. } Assume that $d=d_1d_2$ with $d_1\leq d_2$, and $\theta_\ast\in \mb R^{d_1\times d_2}$ has rank $r\leq \min(d_1,d_2)$. 
Let $\Theta = \l\{ \theta\in \mb R^{d_1\times d_2}: \ \|\theta\|_\ast\leq \|\eta\theta_\ast \|_\ast \r\}$. 
Then the Gaussian mean width of the intersection of a descent cone with a unit ball is bounded as 
$\omega^2\l( D(\Theta,\eta\theta_\ast)\cap \mb S^{d-1} \r) \leq  3r(d_1+d_2 - r)$, see proposition 3.11 in \citep{chandrasekaran2012convex}, hence 
Theorem \ref{master-bound} yields that, with high probability, 
\[
\left\| \widehat{\theta}_m-\eta\theta_* \right\|_2\lesssim \sqrt{\frac{r(d_1+d_2)}{m}}
\] 
as long as the number of observations satisfies $m\gtrsim r(d_1+d_2)$. 
\\
Finally, we derive the corresponding bound from Theorem \ref{master-bound-2}. 
The Gaussian mean width of the unit ball in the nuclear norm is bounded by $2(\sqrt{d_1}+\sqrt{d_2})$, see proposition 10.3 in \citep{vershynin2015estimation}. 
It follows from results in Appendix \ref{app-B} that $\Psi\left(\mb{S}_2\left(\eta\theta_*\right) \right)\leq 4\sqrt{2r}$. 
Theorem \ref{master-bound-2} now implies that with high probability
\[
\left\| \widehat{\theta}_m-\eta\theta_* \right\|_2\lesssim \sqrt{\frac{r(d_1+d_2)}{m}},
\] 
which matches the bound of Theorem \ref{master-bound}.

%########################
\section{Numerical experiments}
%########################

In this section, we demonstrate the performance of proposed robust estimator \eqref{eq:unconstrained-version} for one-bit compressed sensing model. 
The model takes the following form:
\begin{equation}\label{one-bit}
y=sign(\dotp{\mathbf{x}}{\theta_*})+\delta,
\end{equation}
where $\delta$ is the additive noise and the parameter $\theta^*$ is assumed to be $s$-sparse. 
This model is highly non-linear because one can only observe the sign of each measurement.

The 1-bit compressed sensing model was previously discussed extensively in a number of works \citep{plan2014high,ai2014one,plan2016generalized}. 
It was shown that when the measurement vectors are either Gaussian or sub-Gaussian, the Lasso estimator recovers the support of $\theta^*$ with high probability. 
Here, we show that under the heavy-tailed elliptically distributed measurements, our estimator numerically outperforms the standard Lasso estimator
\[
\theta_{\mbox{Lasso}}=\argmin_{\theta\in\mathbb{R}^d}~~\|\mathbf{X}\theta-\mathbf{y}\|_2^2+\lambda \| \theta \|_1,
\]
 while taking the form of a simple soft-thresholding as explained in \eqref{ST-estimator}.

In the first numerical experiment, 
data are simulated in the following way: $\mathbf{x}_1,~\mathbf{x}_2,~\cdots,~\mathbf{x}_{128}\in\mathbb{R}^{512}$ are i.i.d. with spherically symmetric distribution $\mathbf{x}_i\stackrel{d}{=}\mu_i U_i, \ i=1,\ldots,n$. 
%The signal $\theta_*$ is sparse with cardinality equal to 5 chosen uniformly at random from 512 entries.
The random vectors $U_i\in\mathbb{R}^{512}$ are i.i.d. with uniform distribution over the sphere of radius $\sqrt{512}$, and the random variables $\mu_i\in\mathbb{R}$ are also i.i.d., independent of $U_i$ and such that 
\begin{equation}\label{heavy-noise}
\mu_i\stackrel{d}{=}\frac{1}{\sqrt{2c(q)}}(\xi_{i,1}-\xi_{i,2}),
\end{equation}
where $\xi_{i,1}$ and $\xi_{i,2}$,~$i=1,2,\cdots,128$ are i.i.d. with Pareto distribution, meaning that their probability density function is given by
\[
p(t;q) = \frac{q}{(1+t)^{1+q}}I_{\{t>0\}},
\]
$c(q):=\var(\xi)=\frac{q}{(q-1)^2(q-2)}$, and $q=2.1$. 
The true signal $\theta^*$ has sparsity level $s=5$, with index of each non-zero coordinate chosen uniformly at random, and the magnitude having uniform distribution on $[0,1]$. 

Since we can only recover the original signal $\theta^*$ up to scaling, define the relative error for any estimator 
$\hat\theta$ with respect to $\theta^*$ as follows:
\begin{equation}\label{relative-error}
\textrm{Relative~error} = \left|\frac{\hat\theta}{\|\hat\theta\|_2}-\frac{\theta^*}{\|\theta^*\|_2}\right|.
\end{equation}
In each of the following two scenarios, we run the experiment 200 times for both the Lasso estimator and the estimator defined in \eqref{eq:unconstrained-version} with $\|\cdot\|_\m K$ being the $\|\cdot\|_1$ norm. 
We set the truncation level as $\tau = c m^{\frac{1}{2(1+\kappa)}}$, and the values of $c$ and regularization parameter $\lambda$ are obtained via the standard 2-fold cross validation for the relative error \eqref{relative-error}.  
%During each experiment, the estimator is obtained via the standard 2-fold cross validation in Matlab and 
%the regularization parameter $\lambda$ is chosen so as to minimize the relative error \eqref{relative-error} over different $\lambda$'s being tested.
We then plot the histogram of obtained results over 200 runs of the experiment. 

In the first scenario, we set the additive error $\delta_i=0,~i=1,2,\cdots,128$ in the 1-bit model \eqref{one-bit} and plot the histogram in Fig. \ref{fig:Stupendous1}. 
We can see from the plot that the robust estimator \eqref{eq:unconstrained-version} noticeably outperforms the Lasso estimator.

In the second scenario, we set the additive error $\delta_i,~i=1,2,\cdots,128$ to be i.i.d. heavy tailed noise with signal-to-noise ratio (SNR)\footnote{The signal-to-noise ratio (dB) is defined as $\textrm{SNR}:=10\log_{10}(\sigma^2_{\textrm{signal}}/\sigma^2_{\textrm{noise}})$. In our case, since $\langle\mathbf{x}_i,\theta^*\rangle$ can be positive or negative with equal probability, $\sigma^2_{\textrm{signal}}=1$, and thus, $\sigma^2_{\textrm{noise}}=1/10$.} equal to 10dB, so that the noise has the distribution
\[\delta_i\stackrel{d}{=}h_i/\sqrt{10},\]
and $h_i,~i=1,2,\cdots, 128$ are i.i.d. random variables with Pareto distribution, see \eqref{heavy-noise}. 
The results are plotted in Fig. \ref{fig:Stupendous2}. 
The histogram shows that, while performance of the Lasso estimator becomes worse, results of robust estimator \eqref{eq:unconstrained-version} are relatively stable. 

\begin{figure}[htbp]
   \minipage{0.49\textwidth}
   \includegraphics[width=\linewidth]{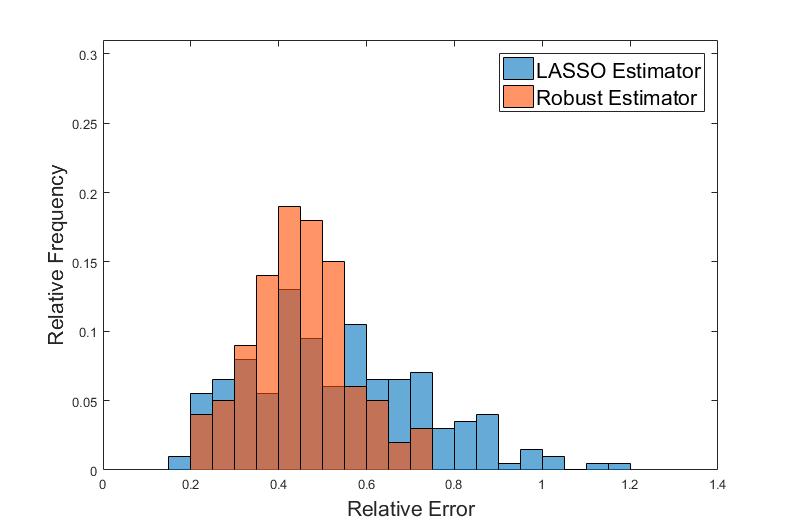} % requires the graphicx package
   \caption{Lasso vs robust estimator without additive noise. \mbox{                }}
   \label{fig:Stupendous1}
   \endminipage\hfill
%\end{figure}
%\begin{figure}[htbp]
 \minipage{0.49\textwidth}
   \includegraphics[width=\linewidth]{200-one-bit-5dB-heavynoise} % requires the graphicx package
   \caption{Lasso vs robust estimator under heavy-tailed noise with signal-to-noise ratio(SNR) equal to $10dB$.}
   \label{fig:Stupendous2}
   \endminipage
\end{figure}

In the second simulation study, the simulation framework similar to the second scenario above, the only difference being the increased sample size $m$. 
The results are plotted in Fig. \ref{fig:awesome_image1}-\ref{fig:awesome_image3} with sample sizes $m=128,~256$ and 512, respectively. 

\begin{figure}[!htb]
\minipage{0.49\textwidth}
  \includegraphics[width=\linewidth]{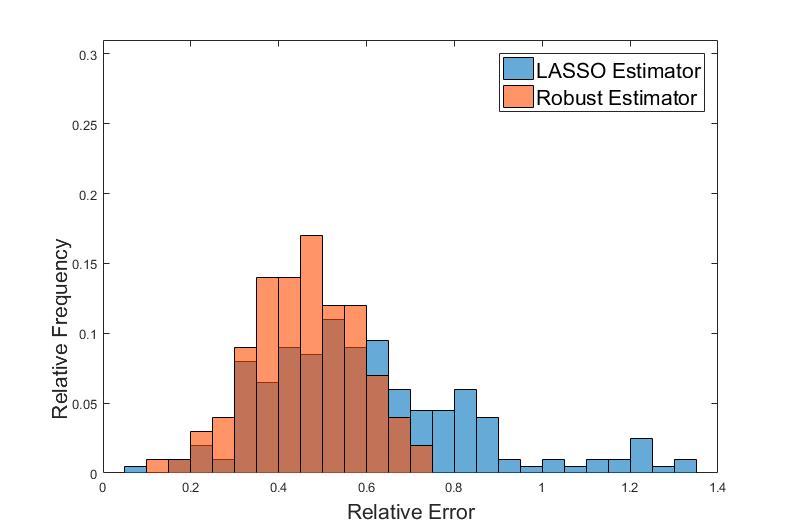}
  \caption{$m=128$}\label{fig:awesome_image1}
\endminipage\hfill
\minipage{0.49\textwidth}
  \includegraphics[width=\linewidth]{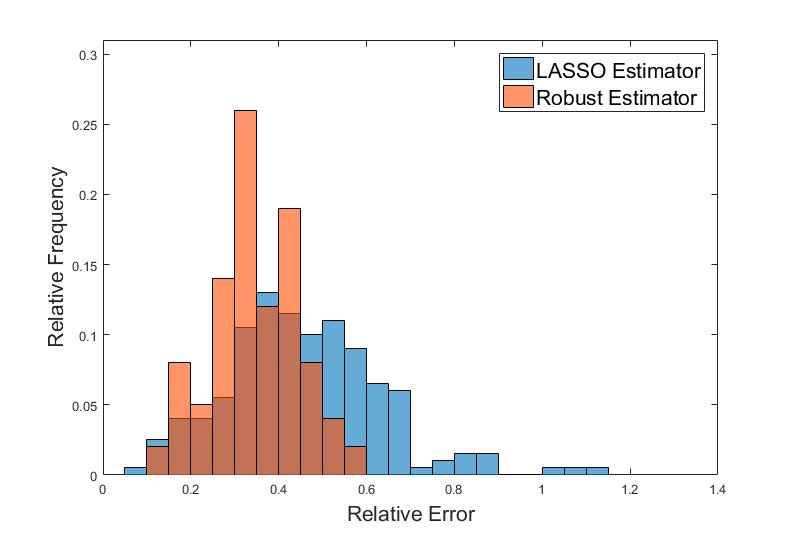}
  \caption{$m=256$}\label{fig:awesome_image2}
\endminipage\hfill
\minipage{0.5\textwidth}%
  \includegraphics[width=\linewidth]{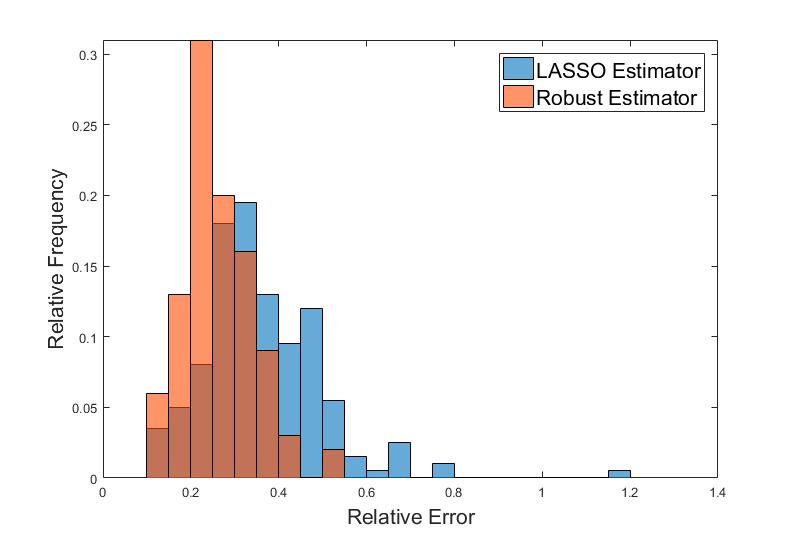}
  \caption{$m=512$}\label{fig:awesome_image3}
\endminipage
\end{figure}

%############
\section{Proofs.}
%############

This  section is devoted to the proofs of Theorems \ref{master-bound} and \ref{master-bound-2}. 

%###################
\subsection{Preliminaries.}
%###################

We recall several useful facts from probability theory that we rely on in the subsequent analysis. 
\\
The following well-known bound shows that the uniform distribution on a high-dimensional sphere enjoys strong concentration properties.
\begin{lemma}[Lemma 2.2 of \cite{convex-geometry-Ball}]
\label{ball-lemma-0}
%Denote the uniform spherical measure on $\mathbb{S}^{d-1}$ %as $d\mathbb{P}$ with %$\int_{\mathbb{S}^{d-1}}d\mathbb{P}=1$, 
Let $U$ have the uniform distribution on $\mathbb{S}^{d-1}$.
%\lcomm{changing here to notation already in use} 
Then for any $\Delta\in(0,1)$ and any fixed $\mathbf{v}\in\mathbb{S}^{d-1}$, 
\[ 
\mb P\left(\langle U,\mathbf{v}\rangle\geq\Delta \right) \leq e^{-d\Delta^2/2}.
\]
%we have
%\[\mathbb{P}\left(\mathbf{x}\in\mathbb{S}^{d-1}:~\langle\ma %thbf{x},\mathbf{v}\rangle\geq\Delta\right)
%\leq e^{-d\Delta^2/2}.\]
\end{lemma}
%##################################################
%\subsection{Bounds for the supremum of a stochastic process and generic chaining.}
%##################################################
\noindent 
Next, we state several useful results from the theory of empirical processes. 
\begin{definition}[$\psi_q$-norm]
For $q \ge 1$, the $\psi_q$-norm of a random variable $\xi\in \mb R$ is given by
\[\|\xi\|_{\psi_q}=\sup_{p\geq1}p^{-\frac{1}{q}}(\expect{|X|^p})^{\frac1p}.\]
Specifically, the cases $q=1$ and $q=2$ are known as the sub-exponential and sub-Gaussian norms respectively.  
%is sub-Gaussian if there exists a positive constant $\kappa$ such that
%$$\sup_{p\geq1}p^{-\frac{1}{2}}(\expect{|X|^p})^{\frac1p}\leq\kappa.$$
%The infimum over all such $\kappa$ is the sub-Gaussian norm with respect to $a$, which we denote as $\|X\|_{\psi_2}$.
We will say that $\xi$ is sub-exponential if $\|\xi\|_{\psi_1}<\infty$, and $X$ is sub-Gaussian if $\|\xi\|_{\psi_2}<\infty$.
\end{definition}
\begin{remark}
\label{norm-justify}
It is easy to check that $\psi_q$-norm is indeed a norm. 
%It can be easily verified that $\psi_q$-norm is indeed a norm. 
%Here we only check the triangle inequality \lcolor{as the homogeneity of $\psi_q$, and that it separates points, are immediate}. For any two real random variables $X$ and $Y$, \lcolor{by the Minkowski inequality,}
%\begin{align*}
%\|X+Y\|_{\psi_q}=\sup_{p\geq1}p^{-\frac{1}{q}}(\expect{|X+Y|^p})^{\frac1p}
%\leq\sup_{p\geq1}p^{-\frac{1}{q}}\left((\expect{|X|^p})^{\frac1p}+(\expect{|Y|^p})^{\frac1p}\right)
%\leq\|X\|_{\psi_q}+\|Y\|_{\psi_q}.
%\end{align*}
\end{remark}
\begin{remark}
A useful property, equivalent to the previous definition of a sub-Gaussian random variable $\xi$, is that there exists a positive constant $C$ such that
\[
\mb P\l( |\xi|\geq u\r)\leq \exp(1-Cu^2).
\]
For the proof, see Lemma 5.5 in \cite{introduction-to-random-matrix}.
\end{remark}

\begin{definition}[sub-Gaussian random vector]
A random vector $\mathbf{x}\in\mathbb{R}^d$ is called sub-Gaussian if there exists $C>0$ such that 
$\|\langle\mathbf{x},\mathbf{v}\rangle\|_{\psi_2}\leq C$ for any $\mathbf{v}\in\mathbb{S}^{d-1}$. 
The corresponding sub-Gaussian norm is then
\[
\|\mathbf{x}\|_{\psi_2}:=\sup_{\mathbf{v}\in\mathbb{S}^{d-1}}\|\langle\mathbf{x},\mathbf{v}\rangle\|_{\psi_2}.
\]
%where $\mathbb{S}^{d-1}$ denotes the unit sphere in $\mathbb{R}^d$.
\end{definition}

Next, we recall the notion of the generic chaining complexity. 
%The following notions are \lcolor{required to implement generic chaining type arguments.} 
Let $(T,d)$ be a metric space. 
We say a collection $\{\mathcal{A}_l\}_{l=0}^{\infty}$ of subsets of $T$ is increasing when $\mathcal{A}_{l} \subseteq \mathcal{A}_{l+1}$ for all $l \ge 0$.
\begin{definition}[Admissible sequence]
An increasing sequence of subsets  $\{\mathcal{A}_l\}_{l=0}^{\infty}$ of $T$ is admissible if $|\mathcal{A}_l|\leq N_l,~\forall l$, where $N_0=1$ and $N_l=2^{2^l},~\forall l\geq1$.
\end{definition}
For each $\mathcal{A}_l$, define the map $\pi_l:T\rightarrow \mathcal{A}_l$ as 
$\pi_l(t)=\textrm{arg}\min_{s\in\mathcal{A}_l}d(s,t),~\forall t\in T$. 
Note that, since each $\mathcal{A}_l$ is a finite set, the minimum is always achieved. 
When the minimum is achieved for multiple elements in $\mathcal{A}_l$, we break the ties arbitrarily.
The generic chaining complexity $\gamma_2$ is defined as
\bea 
\label{eq:def.gamma2}
\gamma_2(T,d):=\inf\sup_{t\in T}\sum_{l=0}^{\infty}2^{l/2}d(t,\pi_l(t)),
\ena
where the infimum is over all admissible sequences. 
%In general, there is no guarantee that this value is always finite. 
The following theorem tells us that $\gamma_2$-functional controls the ``size'' of a Gaussian process. 

\begin{lemma}[Theorem 2.4.1 of \cite{Talagrand-book-2}]
\label{mmt}
Let $\{G(t), \ t\in T\}$ be a centered Gaussian process indexed by the set $T$, and let 
\[
d(s,t)=\expect{(G(s)-G(t))^2}^{1/2},~\forall s,t\in T.
\]
Then, there exists a universal constant $L$ such that 
%which does not depend on \xcolor{any parameter of the Gaussian family}
\[
\frac1L\gamma_2(T,d)\leq\expect{\sup_{t\in T}G(t)}\leq L\gamma_2(T,d).
\]
\end{lemma}
%\begin{remark}
%Note that if we take $G(t)=\langle\mathbf{g},t\rangle,~\forall t\in T\subseteq\mathbb{R}^d$ with $\mathbf{g}\sim\mathcal{N}(0,\mathbf{I}_{d\times d})$, then,
%\[
%d(s,t)=\mb E^{1/2}\langle\mathbf{g},t-s\rangle^2=\|t-s\|_2,~\forall t,s\in T,
%\]
%and hence $\gamma_2(T,d)=\gamma_2(T,\|\cdot\|_2)$.
%\end{remark}
Let $(T,d)$ be a semi-metric space, and let $X_1(t),\cdots,X_m(t)$ be independent stochastic processes indexed by $T$ such that $\mb E|X_j(t)|<\infty$ for all $t\in \mb T$ and $1\leq j\leq m$. 
%For each $t\in T$, we are given an $m$-tuple, where $X_i(t):\Omega'\rightarrow\mathbb{R}$ is an integrable random variable. 
We are interested in bounding the supremum of the empirical process
\begin{equation}\label{empirical}
Z_m(t)=\frac1m\sum_{i=1}^m\l[X_i(t)-\expect{X_i(t)}\r].
\end{equation}
The following well-known symmetrization inequality reduces the problem to bounds on a (conditionally) Rademacher process 
$R_m(t)=\frac 1m\sum_{i=1}^m\varepsilon_i X_i(t),~t\in T$, where $\eps_1,\ldots,\eps_m$ are i.i.d. Rademacher random variables (meaning that they take values $\{-1,+1\}$ with probability $1/2$ each), independent of $X_i$'s.
%Recall that a Rademacher random variable takes the values $\{-1,+1\}$ with probability $1/2$ each.
\begin{lemma}[Symmetrization inequalities]
\label{symmetrization}
\[
\mb E\sup_{t\in T}|Z_m(t)|
\leq 2\mb E\sup_{t\in T}|R_m(t)|,
\]
and for any $u>0$, we have
\[
\mb P\left(\sup_{t\in T}|Z_m(t)|\geq 2\mb E\sup_{t\in T}|Z_m(t)|+u\right)\leq 4\mb P\left(\sup_{t\in T}|R_m(t)|\geq u/2\right).
\]
\end{lemma}
\begin{proof}
See Lemmas 6.3 and 6.5 in \citep{Talagrand-book}
\end{proof}
	\begin{comment}
\begin{remark}[Measurability issue]\label{remark-measurability}
The precise meaning of~~$\expect{\sup_{t\in T}X(t)}$ in Lemma \ref{symmetrization} is not clear if $T$ is uncountable.
In general, for an uncountable index set $T$, the function $\sup_{t\in T}X(t)$ for a collection of random variables $\{X(t)\}_{t\in T}$ might not be measurable. Counter examples exist even under the case where the function $X(\cdot)$ is jointly measurable on the product space $(\Omega\times T,~\mathcal{E}\otimes\mathcal{T})$ (first constructed by Luzin and Suslin), where 
$\mathcal{E}$ and $\mathcal{T}$ are Borel $\sigma$-algebras on $\Omega$ and $T$ respectively. However, when $T$ is a measurable subset of $\mathbb{R}^d$(which is the case we are interested in) and $X(\cdot)$ is jointly measurable on $(\Omega\times T,~\mathcal{E}\otimes\mathcal{T})$, one can show that the $\sup_{t\in T}X(t)$ is always measurable. 

Indeed, 
$\sup_{t\in T}X(t)$ is measurable if and only if the set $\{\sup_{t\in T}X(t)> c\}\in\mathcal{E}$ for any $c\in\mathbb{R}$. On the other hand, $\{\sup_{t\in T}X(t)> c\}=P_{\Omega}\{X(\cdot)> c\}$, where 
for any set $A\in\Omega\times T$, 
$P_{\Omega}A:=\{\omega\in\Omega:(\omega,t)\in A\}$ is the projection of the set $A$ onto $\Omega$. Then, the measurability comes from the following theorem in \cite{Cohn}: If $(\Omega, \mathcal{E})$ is a measurable space and $T$ is a Polish space, then, the projection onto $\Omega$ of any product measurable subset of $\Omega\times T$ is also measurable. 
\end{remark}
	\end{comment}
Finally, we recall Bernstein's concentration inequality. 
\begin{lemma}[Bernstein's inequality]
\label{Bernstein}
Let $X_1,\cdots,X_m$ be a sequence of independent centered random variables. 
Assume that there exist positive constants $\sigma$ and $D$ such that for all integers $p\geq 2$
\[
\frac1m\sum_{i=1}^m\expect{|X_i|^p}\leq\frac{p!}{2}\sigma^2D^{p-2},
\]
then
\[
\mb P\left(\left|\frac1m\sum_{i=1}^m X_i\right|\geq\frac{\sigma}{\sqrt{m}}\sqrt{2u}+\frac{D}{m}u\right)
\leq2\exp(-u).
\]
In particular, if $X_1,\cdots,X_m$ are all sub-exponential random variables, then $\sigma$ and $D$ can be chosen as 
$\sigma=\frac{1}{m}\sum_{i=1}^m\|X_i\|_{\psi_1}$ and $D=\max\limits_{i=1\ldots m}\|X_i\|_{\psi_1}$. 
\end{lemma}

%#########################
\subsection{Roadmap of the proof of Theorem \ref{master-bound}.}
%#########################
%\marginpar{\Stas{I'm rewriting this using the notation of previous section}}

We outline the main steps in the proof of Theorem \ref{master-bound}, and postpone some technical details to sections \ref{section:consistency} and \ref{section:chaining}.\\
%The proof is based on the estimation of the difference $L^0(\widehat \theta_m) - L^0(\eta \widehat \theta_*)$. 
As it will be shown below in Lemma \ref{lemma:mean-consistency}, 
$\argmin\limits_{\theta\in \Theta}L^0(\theta) = \eta \theta_\ast$ for $\eta = \mb E\l( \dotp{y\mathbf{x}}{\theta_\ast}\r)$ and $L^0(\widehat \theta_m) - L^0(\eta\theta_\ast) = \|\widehat\theta_m - \eta\theta_\ast\|_2^2$, hence 
\begin{align}
\|\widehat\theta_m - \eta\theta_\ast\|_2^2 & =
L^\tau(\widehat \theta_m) - L^\tau(\eta\theta_\ast) + \l( L^0(\widehat\theta_m) - L^\tau(\widehat\theta_m) - L^0(\eta\theta_\ast) + L^\tau(\eta\theta_\ast)  \r) \nonumber\\
&
=L^\tau(\widehat\theta_m) - L^\tau(\eta\theta_\ast) +(L_m^\tau(\widehat\theta_m) - L_m^\tau(\eta\theta_\ast)) 
\nonumber\\
& 
\quad - (L_m^\tau(\widehat\theta_m) - L_m^\tau(\eta\theta_\ast)) - 2\mb E_m \dotp{y\mathbf{x} - \widetilde q \widetilde U}{\widehat\theta_m - \eta\theta_\ast}, \label{split-bound}
\end{align}
where $\mb E_m(\cdot)$ stands for the conditional expectation given $(\mathbf{x}_i,y_i)_{i=1}^m$, and where we used the equality 
$L^0(\widehat\theta_m) - L^\tau(\widehat\theta_m) - L^0(\eta\theta_\ast) + L^\tau(\eta\theta_\ast)  = 
-2 \mb E_m\left( \dotp{y\mathbf{x} - \widetilde q \widetilde U}{\widehat\theta_m - \eta\theta_\ast}\right)
$ in the last step. 
Since $\widehat\theta_m$ minimizes $L_m^\tau$, $L_m^\tau(\widehat\theta_m) - L_m^\tau(\eta\theta_\ast)\leq 0$, and
\begin{align*}
\|\widehat\theta_m - \eta\theta_\ast\|_2^2 \leq & \,
\frac{2}{m}\sum_{i=1}^m \l( \dotp{ \widetilde q_i \widetilde U_i}{\widehat\theta_m - \eta\theta_\ast}    - \mb E_m \l(\dotp{\widetilde q \widetilde U}{\widehat\theta_m - \eta\theta_\ast}\r)\r)  \\
& 
-2\mb E_m\left(  \dotp{y\mathbf{x} - \widetilde q \widetilde U}{\widehat\theta_m - \eta\theta_\ast} \right).
\end{align*}
Note that $\widehat\theta_m - \eta\theta_\ast \in D(\Theta,\eta\theta_\ast)$; 
dividing both sides of the inequality by $\|\widehat\theta_m - \eta\theta_\ast\|_2$, we obtain
\begin{align}
\label{eq:main}
&
\|\widehat\theta_m - \eta\theta_\ast\|_2 \leq 
\sup_{\mathbf{v}\in D(\Theta,\eta\theta_\ast)\cap \mb S^{d-1}}\l| \frac{2}{m}\sum_{i=1}^m \dotp{\widetilde q_i\widetilde U_i}{\mathbf{v}} - 
\mb E  \dotp{\widetilde q\widetilde U}{\mathbf{v}}\r| +
2\sup_{\mathbf{v}\in \mb S^{d-1}}\mb E \dotp{y\mathbf{x} - \widetilde q \widetilde U}{\mathbf{v}}.
\end{align}
To get the desired bound, it remains to estimate two terms above.  
%\lcomm{removed `in the sum ' as the sum intended, the $+$ can be confused with the $\sum$}
The bound for the first term is implied by Lemma \ref{concentration-multiplier}: setting $T=D(\Theta,\eta\theta_\ast)\cap \mb S^{d-1}$, and observing that the diameter $\Delta_d(T) := \sup_{t\in T}\|t\|_2=1$, we get that with probability $\geq 1 - ce^{-\beta/2}$,
\[
\sup_{\mathbf{v}\in D(\Theta,\eta\theta_\ast)\cap \mb S^{d-1}}
\l| \frac{2}{m}\sum_{i=1}^m \dotp{\widetilde q_i\widetilde U_i}{\mathbf{v}} - 
\mb E  \dotp{\widetilde q\widetilde U}{\mathbf{v}}\r| 
\leq 
C\frac{(\omega(T)+1)\beta}{\sqrt{m}}.
\]
To estimate the second term, we apply Lemma \ref{lemma:truncation-bias}: 
\[
2\sup_{\mathbf{v}\in \mb S^{d-1}}\mb E  \dotp{y\mathbf{x} - \widetilde q \widetilde U}{\mathbf{v}} \
\leq
\frac{\tilde C}{\sqrt m}.
\]
Result of Theorem \ref{master-bound} now follows from the combination of these bounds.
\qed 

%\lcomm{not clear to the reader if this is the actual proof, with reference to lemmas, or just an outline of the proof. I think the former? Clarification needed, and e.g., adding in end of proof symbols and such, if so.}\xcomm{These are actual proofs.}

%#####################################################
\subsection{Roadmap of the proof of Theorem \ref{master-bound-2}.}
%#####################################################

Once again, we will present the main steps while skipping the technical parts. 
Lemma \ref{lemma:mean-consistency} implies that 
$\argmin\limits_{\theta\in \Theta}L^0(\theta) = \eta \theta_\ast$ for 
$\eta = \mb E \dotp{y\mathbf{x}}{\theta_\ast}$ and 
\[
L^0(\widehat \theta_m^\lambda) - L^0(\eta\theta_\ast) = \|\widehat\theta_m^\lambda - \eta\theta_\ast\|_2^2.
\] 
Thus, arguing as in \eqref{split-bound},
\begin{align*}
\|\widehat\theta_m^\lambda - \eta\theta_\ast\|_2^2 & =
L^\tau(\widehat\theta_m^\lambda) - L^\tau(\eta\theta_\ast) +(L_m^\tau(\widehat\theta_m^\lambda) - L_m^\tau(\eta\theta_\ast)) \\
& 
\quad - (L_m^\tau(\widehat\theta_m^\lambda) - L_m^\tau(\eta\theta_\ast)) - 2\mb E_m  \dotp{y\mathbf{x} - \widetilde q \widetilde U}{\widehat\theta_m^\lambda - \eta\theta_\ast}.
\end{align*}
Since $\widehat{\theta}^{\lambda}_m$ is a solution of problem \eqref{eq:unconstrained-version}, it follows that
\begin{align*}
L_m^{\tau}(\theta_m^{\lambda})+\lambda\l\| \theta_m^{\lambda} \r\|_{\mathcal{K}}\leq
L_m^{\tau}\left(\eta\theta_*\right)+\lambda\l \|\eta\theta_* \r\|_{\mathcal{K}},
\end{align*}
which further implies that
\begin{align}
\|\widehat{\theta}_m^\lambda-\eta\theta_*\|_2^2 \leq
&
\frac{2}{m}\sum_{i=1}^m \l( \dotp{\widetilde q_i\widetilde U_i}{\widehat\theta_m^\lambda - \eta\theta_\ast}  - 
\mb E_m \dotp{\widetilde q\widetilde U}{\widehat\theta_m^\lambda - \eta\theta_\ast} \r) 
- 2\mb E_m  \dotp{y\mathbf{x} - \widetilde q \widetilde U}{\widehat\theta_m^\lambda - \eta\theta_\ast}  
\nonumber\\
&
+\lambda\left(\|\eta\theta_*\|_{\mathcal{K}}-\|\widehat{\theta}_m^\lambda\|_{\mathcal{K}}\right)\nonumber\\
=&
  \dotp{\frac{2}{m}\sum_{i=1}^m\widetilde q_i\widetilde U_i -\expect{\widetilde q\widetilde U}}{\widehat\theta_m^\lambda - \eta\theta_\ast} - 2\mb E_m \dotp{y\mathbf{x} - \widetilde q \widetilde U}{\widehat\theta_m^\lambda - \eta\theta_\ast}  
 \nonumber\\
&
+\lambda\left(\|\eta\theta_*\|_{\mathcal{K}}-\|\widehat{\theta}_m^\lambda\|_{\mathcal{K}}\right).
\label{sth-1}
\end{align}
Letting $\|\cdot\|_{\mathcal{K}}^*$ be the dual norm of $\|\cdot\|_{\mathcal{K}}$ 
(meaning that $\|\mf{x}\|_{\m K}^\ast = \sup\l\{ \dotp{\mf{x}}{\mf{z}}, \ \|\mf{z}\|_\m K\leq 1 \r\}$),
%\footnote{Recall that if $\mathcal{X}$ is a Banach space with norm $\|\cdot\|$ then the dual norm $\|\cdot\|^*$, over the dual space $\mathcal{X}^*=\{g:~g~\textrm{is a bounded linear functional}~\mathcal{X}\rightarrow\mathbb{R}\}$, 
%is defined as $\|g\|^*:=\sup\{g(x):~\|x\|\leq1,~x\in\mathcal{X}\}$.} 
%`trace duality' is immediate, and here yields for
the first term in \eqref{sth-1} can be estimated as 
\begin{align}\label{sth-2}
\dotp{\frac{1}{m}\sum_{i=1}^m\widetilde q_i\widetilde U_i -\expect{\widetilde q\widetilde U}}{\widehat\theta_m^\lambda - \eta\theta_\ast}
\leq\left\|\frac{1}{m}\sum_{i=1}^m\widetilde q_i\widetilde U_i -\expect{\widetilde q\widetilde U}\right\|_{\mathcal{K}}^*\cdot\|\widehat{\theta}_m^\lambda-\eta\theta_*\|_{\mathcal{K}}.
\end{align}
Since 
\[
\left\|\frac{1}{m}\sum_{i=1}^m\widetilde q_i\widetilde U_i -\expect{\widetilde q\widetilde U}\right\|_{\mathcal{K}}^*=\sup_{\|t\|_{\mathcal{K}}\leq1}\dotp{\frac{1}{m}\sum_{i=1}^m\widetilde q_i\widetilde U_i -\expect{\widetilde q\widetilde U}}{t},
\]
lemma \ref{concentration-multiplier} applies 
%\lcomm{need $T$ here to be compact to invoke this lemma. We have boundedness, at least...} 
with $T=\mathcal{G}:=\{\mathbf{x}\in\mathbb{R}^d:~\|\mathbf{x}\|_{\mathcal{K}}\leq1\}$. 
Together with an observation that
$\Delta_d(T)\leq\sup_{t\in T}\|t\|_{\mathcal{K}}=1$ (due to the assumption $\|\mathbf{v}\|_2\leq\|\mathbf{v}\|_{\mathcal{K}}, \ \forall \mathbf{v}\in\mathbb{R}^d$), this yiels
\[
\mb P\left( 
\sup_{\|t\|_{\mathcal{K}}\leq1}\left|\dotp{\frac{1}{m}\sum_{i=1}^m\widetilde q_i\widetilde U_i -\expect{\widetilde q\widetilde U}}{t}\right|
\geq C'\frac{\left(\omega(\mathcal{G})+1\right)\beta}{\sqrt{m}} \right)
\leq c'e^{-\beta/2},
\]
for any $\beta\geq8$ and some constants $C', c>0$. For the second term in \eqref{sth-1}, we use Lemma \ref{lemma:truncation-bias} to obtain
\[
2\mb E_m \dotp{y\mathbf{x} - \widetilde q \widetilde U}{\widehat\theta_m^\lambda - \eta\theta_\ast} 
\leq\frac{C''}{\sqrt{m}}\|\widehat\theta_m^\lambda-\eta\theta_*\|_2
\leq\frac{C''}{\sqrt{m}}\|\widehat\theta_m^\lambda-\eta\theta_*\|_{\mathcal{K}},\]
for some constant $C''>0$, where we have again applied the inequality $\|\mathbf{v}\|_2\leq\|\mathbf{v}\|_{\mathcal{K}}$.
Combining the above two estimates gives that with probability at least $1-ce^{-\beta/2}$,
\begin{align}
\label{sth-3}
\|\widehat{\theta}_m^\lambda-\eta\theta_*\|_2^2
\leq C\frac{\left(\omega(\mathcal{G})+1\right)\beta}{\sqrt{m}}\|\widehat{\theta}_m^\lambda-\eta\theta_*\|_{\mathcal{K}}+\lambda\left(\|\eta\theta_*\|_{\mathcal{K}}-\|\widehat{\theta}_m^\lambda\|_{\mathcal{K}}\right),
\end{align}
for some constant $C>0$ and any $\beta\geq8$. 
Since $\lambda\geq 2C \left(\omega(\mathcal{G})+1\right)\beta/\sqrt{m}$ by assumption, and the right hand side of \eqref{sth-3} is nonnegative, it follows that
\[
\frac12\|\widehat{\theta}_m^\lambda-\eta\theta_*\|_{\mathcal{K}}+\|\eta\theta_*\|_{\mathcal{K}}-\|\widehat{\theta}_m^\lambda\|_{\mathcal{K}}\geq 0.
\]
This inequality implies that $\widehat{\theta}_m^\lambda-\eta\theta_*\in\mb{S}_2(\eta\theta_*)$. 
Finally, from \eqref{sth-3} and the triangle inequality,
\begin{align*}
\|\widehat{\theta}_m^\lambda-\eta\theta_*\|_2^2
\leq \frac32\lambda\|\widehat{\theta}_m^\lambda-\eta\theta_*\|_{\mathcal{K}}.
\end{align*}
Dividing both sides by $\|\widehat{\theta}_m^\lambda-\eta\theta_*\|_2$ gives
\begin{align*}
\|\widehat{\theta}_m^\lambda-\eta\theta_*\|_2
\leq \frac32\lambda\frac{\|\widehat{\theta}_m^\lambda-\eta\theta_*\|_{\mathcal{K}}}{\|\widehat{\theta}_m^\lambda-\eta\theta_*\|_2}
\leq\frac32\lambda\cdot\Psi\left(\mb{S}_2(\eta\theta_*)\right).
\end{align*}
This finishes the proof of Theorem \ref{master-bound-2}.

%######################
\subsection{Bias of the truncated mean.}
\label{section:consistency}
%######################

%The first lemma shows that without truncation, the estimate is consistent in the expectation sense.
The following lemma is motivated by and is similar to Theorem 2.1 in \citep{li1989regression}.
\begin{lemma}
\label{lemma:mean-consistency}
Let $\eta=\mb E \langle y\mathbf{x},\theta_\ast\rangle$. 
Then 
\[
\eta\theta_\ast = \argmin_{\theta\in \Theta} L^0(\theta),
\]
and for any $\theta\in \Theta$,
\[
L^0(\theta) - L^0(\eta\theta_\ast) = \|\theta - \eta\theta_\ast\|_2^2.
\]
%\[
%\expect{\langle\widetilde{U}_i,t\rangle q_i}=\eta\langle\theta_*,t\rangle,
%\]
%for any $t\in\mathbb{S}^{d-1}$.
\end{lemma}
\begin{proof}
Since $y=f(\dotp{\mathbf{x}}{\theta_*},\delta)$, we have that for any $\theta\in\mathbb{R}^d$ 
\begin{align*}
\mb E\dotp{y\mathbf{x}}{\theta}
=&\mb E \langle\mathbf{x},\theta\rangle f(\langle\mathbf{x},\theta_*\rangle,\delta)\\
=&\mb E
\expect{\langle\mathbf{x},\theta\rangle f(\langle\mathbf{x},\theta_*\rangle,\delta)
~|~\langle\mathbf{x},\theta_*\rangle,\delta}\\
=&\mb E
\mb E \l(\langle\mathbf{x},\theta\rangle 
~|~\langle\mathbf{x},\theta_*\rangle \r) \cdot f(\langle\mathbf{x},\theta_*\rangle,\delta)\\
=&\mb E \Big( \langle \theta_*,\theta\rangle\langle\mathbf{x},\theta_*\rangle
 f(\langle\mathbf{x},\theta_*\rangle,\delta) \Big) \\
 =&\eta\langle\theta_*,\theta\rangle,
\end{align*}
where the third equality follows from the fact that the noise $\delta$ is independent of the measurement vector $\mathbf{x}$, the second to last equality from the properties of elliptically symmetric distributions (Corollary \ref{elliptical-corollary}), 
and the last equality from the definition of $\eta$. 
Thus,
%\lcomm{some may need this friendly reminder}
\begin{align*}
L^0(\theta)=&\|\theta\|_2^2-2\expect{\dotp{y\mathbf{x}}{\theta}} =\|\theta\|_2^2-2\eta\langle\theta_*,\theta\rangle
=\|\theta-\eta\theta_*\|_2^2-\|\eta\theta_*\|_2^2,
\end{align*}
which is minimized at $\theta=\eta\theta^*$. Furthermore, $L^0(\eta\theta^*)=-\|\eta\theta_*\|_2^2$, hence
\[
L^0(\theta) - L^0(\eta\theta_\ast) = \|\theta - \eta\theta_\ast\|_2^2,
\]
finishing the proof.
\end{proof}

Next, we estimate the ``bias term'' $\sup_{\mathbf{v}\in \mb S^{d-1}}\mb E \dotp{y\mathbf{x} - \widetilde q \widetilde U}{\mathbf{v}}$ in inequality \eqref{eq:main}. 
In order to do so, we need the following preliminary result.
%Next, we need to show that the truncation procedure produces mean bias on the order $\mathcal{O}(1/\sqrt{m})$. 
\begin{lemma}
\label{lemma:ball}
If $\mathbf{x}\sim\mathcal{E}(0,~\mathbf{I}_{d\times d},~F_{\mu})$, 
then the unit random vector $\mathbf{x}/\|\mathbf{x}\|_2$ is uniformly distributed over the unit sphere $\mathbb{S}^{d-1}$. Furthermore, $\widetilde{U}=\sqrt{d}\mathbf{x}/\|\mathbf{x}\|_2$ is a sub-Gaussian random vector with sub-Gaussian norm $\|\widetilde{U}\|_{\psi_2}$ independent of the dimension $d$.
\end{lemma}
\begin{proof}
First, we use decomposition \eqref{elliptical-definition} for elliptical distribution together with our assumption that $\mathbf{\Sigma}$ is the identity matrix, to write $\mathbf{x}\stackrel{d}{=}\mu U$, which implies that
\[
\mathbf{x}/\|\mathbf{x}\|_2
\stackrel{d}{=}\textrm{sign}(\mu)U/\| U\|_2
=\textrm{sign}(\mu)U \stackrel{d}{=}U,
\]
with the final distributional equality holding as $\mathbb{S}^{d-1}$, and hence its uniform distribution, is invariant with respect to reflections across any hyperplane through the origin.
%\lcomm{overkill, but convincing.} 

To prove the second claim, it is enough to show that 
$\left\|\dotp{\widetilde{U}}{\mathbf{v}}\right\|_{\psi_2}\leq C,~\forall \mathbf{v}\in\mathbb{S}^{d-1}$ 
with constant $C$ independent of $d$. 
By the first claim and Lemma \ref{ball-lemma-0}, we have 
\[
\mb P\left( 
\langle\mathbf{x},\mathbf{v}\rangle/\|\mathbf{x}\|_2\geq \Delta \right)
\leq e^{-d\Delta^2/2},~\forall \mathbf{v}\in\mathbb{S}^{d-1}.
\]
Choosing $\Delta=u/\sqrt{d}$ gives
\[
\mb P \left( 
\dotp{\widetilde{U}}{\mathbf{v}}\geq u \right)\leq e^{-u^2/2},~\forall \mathbf{v}\in\mathbb{S}^{d-1},~\forall u>0.
\]
By an equivalent definition of sub-Gaussian random variables (Lemma 5.5 of \cite{introduction-to-random-matrix}), this inequality implies that 
$\left\|\dotp{\widetilde{U}}{\mathbf{v}}\right\|_{\psi_2}\leq C$, hence finishing the proof.
\end{proof}
With the previous lemma in hand, we now establish the following result.

\begin{lemma}
\label{lemma:truncation-bias}
Under the assumptions of Theorem \ref{master-bound}, there exists a constant $C=C(\kappa,\phi)>0$ such that
\[
\left|\mb E\dotp{y\mathbf{x}-\widetilde{q}\widetilde{U}}{\mathbf{v}}\right|\leq C/\sqrt{m},
\]
for all $\mathbf{v}\in\mathbb{S}^{d-1}$.
\end{lemma}
\begin{proof}
%\Stas{Let $q_\eps=\frac{1+\eps}{\eps}$. Then H\"{o}lder's inequality implies that }
By \eqref{transformation}, we have that $y\mathbf{x}=q\widetilde{U}$, thus the claim is equivalent to
\[
\left|\expect{\dotp{\widetilde{U}}{\mathbf{v}}(\widetilde{q}-q)}\right|\leq C/\sqrt{m}.
\]
Since $\widetilde{q}=\textrm{sign}(q)(|q|\wedge\tau)$, we have $|\widetilde{q}-q|=(|q|-\tau){\bf 1}(|q| \ge \tau) \le |q|{\bf 1}(|q| \ge \tau)$, and it follows that
\begin{align*}
\left| \mb E\dotp{\widetilde{U}}{\mathbf{v}} (\widetilde{q}-q)\right|
\leq& \mb E\left|\dotp{\widetilde{U}}{\mathbf{v}} (\widetilde{q}-q)\right| \\
\leq& \mb E \l( \left|\dotp{\widetilde{U}}{\mathbf{v}} q\right|\cdot\mathbf{1}_{\{|q|\geq\tau\}} \r) \\
\leq&\expect{\left|\dotp{\widetilde{U}}{\mathbf{v}} q\right|^2}^{1/2} \mb P\l( |q|\geq\tau \r)^{1/2}\\
\leq&\expect{\left|\dotp{\widetilde{U}}{\mathbf{v}}\right|^{\frac{2(1+\kappa)}{\kappa}}}^{\frac{\kappa}{2(1+\kappa)}}
\expect{|q|^{2(1+\kappa)}}^{\frac{1}{2(1+\kappa)}} \mb P\l( |q|\geq\tau \r)^{1/2},
%\leq &\Stas{ \mb E^{1/3}|\dotp{\widetilde a_i }{t}q_i|^3 \l(\Pr\l( |q_i|>\tau \r)\r)^{2/3}} \\
%\leq & \Stas{ \mb E^{1/(3(1+\eps))} |q_i|^{3(1+\eps)}  \mb E^{1/3q_\eps} |\dotp{\widetilde a_i}{t}|^{3q_\eps} \l(\Pr\l( |q_i|>\tau \r)\r)^{2/3}  }
%\leq&\expect{|\langle\widetilde{U}_i,t\rangle q_i|^2}^{1/2}\cdot Pr[|q_i|\geq\tau]^{1/2}\\
%\leq&\expect{|\langle\widetilde{U}_i,t\rangle|^4}^{1/4} \expect{q_i^4}^{1/4}\cdot Pr[|q_i|\geq\tau]^{1/2},
\end{align*}
where the second to last inequality uses Cauchy-Schwarz, and the last inequality follows from H\"{o}lder's inequality.

For the first term, by Lemma \ref{lemma:ball}, $\widetilde{U}$ is sub-Gaussian with $\|\widetilde{U}\|_{\psi_2}$ independent of $d$. Thus, by the definition of the $\|\cdot\|_{\psi_2}$ norm and the fact that ${\bf v} \in \mathbb{S}^{d-1}$,
\[
\expect{ \left|\dotp{\widetilde U}{\mathbf{v}}\right|^{\frac{2(1+\kappa)}{\kappa}}}^{\frac{\kappa}{2(1+\kappa)}}
\leq \sqrt{\frac{2(1+\kappa)}{\kappa}} \|\widetilde U \|_{\psi_2}.
\]

Recall that $\phi= \mb E |q|^{2(1+\kappa)}$.  
Then, the second term is bounded by $\phi^{\frac{1}{2(1+\kappa)}}$. 
For the final term, since
$\tau=m^{\frac{1}{2(1+\kappa)}}$, Markov's inequality implies that
\begin{align*}
  \l( \mb P \l( |q|>\tau \r) \r) ^{1/2}\leq\left(\frac{\mb E |q|^{2(1+\kappa)}}{\tau^{2(1+\kappa)}}\right)^{1/2}
 \leq\frac{\phi^{1/2}}{\sqrt{m}}.
\end{align*}
Combining these inequalities yields
\[
\left| \mb E\dotp{y\mathbf{x}-\widetilde{q}\widetilde{U}}{\mathbf{v}}\right|
\leq 
\frac{\sqrt{\frac{2(1+\kappa)}{\kappa}} \|\widetilde U \|_{\psi_2}\phi^{\frac{2+\kappa}{2(1+\kappa)}}}{\sqrt{m}}:=C(\kappa,\phi)/\sqrt{m},\]
completing the proof.
\end{proof}

%\Stas{I wonder if we can truncate on a ``higher'' level, say $m^{1/2-\eps}$, to get away with even weaker moment assumptions. 
%This is mainly related to chaining bounds of course.}

\subsection{Concentration via generic chaining.}
\label{section:chaining}
%In this section, we choose $d(\mathbf{x},\mathbf{y})=\|\mathbf{x}-\mathbf{y}\|_2$ and the Talagrand functional $\gamma_2(T):=\gamma_2(T,\|\cdot\|_2)$.
In the following sections, we will use $c,C,C',C''$ to denote constants that are either absolute, or depend on underlying parameters $\kappa$ and $\phi$ (in the latter case, we specify such dependence). 
To make notation less cumbersome, constants denoted by the same letter ($c,C,C'$, etc.) might be different in various parts of the proof.  

The goal of this subsection is to prove the following inequality:
\begin{lemma}
\label{concentration-multiplier}
Suppose $\widetilde U_i$ and $\widetilde q_i$ are as defined according to \eqref{transformation} and \eqref{transformation-2} respectively. 
Then, for any bounded subset $T\subset\mathbb{R}^d$,
%\lcomm{compact sets are closed, so automatically Borel measurable} 
\begin{align*}
\mb P\left( \sup_{t\in T}\left|\frac1m\sum_{i=1}^m\dotp{\widetilde{U}_i}{t} \widetilde{q}_i
-\expect{\dotp{\widetilde{U}}{t} \widetilde{q}}\right|
\geq C\frac{(\omega(T)+\Delta_d(T))\beta}{\sqrt{m}} \right)
\leq ce^{-\beta/2},
\end{align*}
for any $\beta\geq8$, a positive constant $C=C(\kappa,\phi)$ and an absolute constant $c>0$. 
Here
%\lcomm{labeled for later referencing}
\bea 
\label{def:Deltad}
\Delta_d(T):=\sup_{t\in T}\|t\|_2.
\ena
\end{lemma}
The main technique we apply is the generic chaining method developed by M. Talagrand \citep{Talagrand-book-2} for bounding the supremum of stochastic processes. 
Recently, \cite{Mendelson-1} and \cite{tail-bound-chaining} advanced the technique to obtain a sharp bound for supremum of processes index by squares of functions. 
%\footnote{A quadratic process takes the form $\frac{1}{m}\sum_{i=1}^m\left(X_i^2(t)-\expect{X_i^2(t)}\right),~\forall t\in T$, where each $X_i(t)$ is sub-Gaussian with uniformly bounded sub-Gaussian norm.}. 
More recently, \cite{Mendelson-2} proved a concentration result for the supremum of multiplier processes
%\footnote{A multiplier process takes the form $\frac{1}{m}\sum_{i=1}^m\left(w_iX_i(t)-\expect{w_iX_i(t)}\right),~\forall t\in T$, where $\{w_i\}_{i=1}^m$ is a sequence of random variables which does not have to be independent of $\{X_i(t)\}_{i=1}^m$.} 
under weak moment assumptions. In the current work, we show that exponential-type concentration inequalities for multiplier processes, such as the one in Lemma \ref{concentration-multiplier}, are achievable by applying truncation under a bounded $2(1+\kappa)$-moment assumption.

Define
\begin{align*} 
\overline{Z}(t)=&\frac1m\sum_{i=1}^m\dotp{\widetilde{U}_i}{t} \widetilde{q}_i
-\expect{\dotp{\widetilde{U}}{t} \widetilde{q}},\\
Z(t)=&\frac1m\sum_{i=1}^m\varepsilon_i \widetilde{q}_i\dotp{\widetilde{U}_i}{t},~\forall t\in T,
\end{align*}
where $T$ is a bounded set in $\mathbb{R}^d$ and $\{\varepsilon_i\}_{i=1}^m$ is a sequence i.i.d. Rademacher random variables taking values $\pm 1$  with probability $1/2$ each, and independent of $\{\widetilde{U}_i,\widetilde{q}_i, \ i=1, \ldots,m\}$. 
Result of Lemma \ref{concentration-multiplier} easily follows from the following concentration inequality:
\begin{lemma}
For any $\beta\geq8$,
\begin{equation}\label{major-criterion}
\mb P\left[\sup_{t\in T}\left|Z(t)\right|
\geq C\frac{(\omega(T)+\Delta_d(T))\beta}{\sqrt{m}}\right]
\leq ce^{-\beta/2},
\end{equation}
where $C=C(\kappa,\phi)$ is another constant possibly different from that of Lemma \ref{concentration-multiplier}, and $c>0$ is an absolute constant.
\end{lemma}
To deduce the inequality of Lemma \ref{concentration-multiplier}, we first apply the symmetrization inequality (Lemma \ref{symmetrization}), followed by Lemma \ref{basic-inequality} with $\beta_0=8$. 
It implies that
\[
\expect{\sup_{t\in T}\left|\overline{Z}(t)\right|}
\leq2\expect{\sup_{t\in T}\left|Z(t)\right|}\leq2C\left(8+2ce^{-4}\right)\frac{\omega(T)+\Delta_d(T)}{\sqrt{m}}.
\]
Application of the second bound of the symmetrization lemma with $u=2C(\omega(T)+\Delta_d(T))\beta/\sqrt{m}$ and \eqref{major-criterion} completes the proof of Lemma \ref{concentration-multiplier}.

%\lcomm{need range of $\beta$ in \eqref{major-criterion} to apply Lemma \ref{basic-inequality}. Is the lemma applied with $\beta_0=0$, as there is no constant term? Would need inequality to hold for all $\beta \ge 0$. Should state above when $C$ is first introduced that it does not necessarily take the same values at all occurrences.}
%\lcomm{*Lemma \ref{concentration-multiplier} concerns an general compact $T$, while these statements, which are claimed to imply this lemma, have a specific choice for $T$. Note also the $T$ on the left hand side of \eqref{major-criterion} has not been specified (to be the `choice of $T$' appearing also on the right.}
%\xcomm{The specified choice of $T$ was a typo, should state for a general $T$.}
%\lcomm{the material in this area should be inside of a formal proof}

It remains to justify \eqref{major-criterion}.
We start by picking an arbitrary point $t_0\in T$ such that there exists an admissible sequence $\{t_0\}=\mathcal{A}_0\subseteq\mathcal{A}_1\subseteq\mathcal{A}_2\subseteq\cdots$ satisfying
\begin{equation}\label{inter-1}
\sup_{t\in T}\sum_{l=0}^\infty2^{l/2}\|\pi_l(t)-t\|_2\leq 2\gamma_2(T),
\end{equation}
where we recall that $\pi_l$ is the closest point map from $T$ to $\mathcal{A}_l$ and the factor 2 is introduced so as to deal with the case where the infimum in the definition \eqref{eq:def.gamma2} of $\gamma_2(T)$ is not achieved. Then, write $Z(t)-Z(t_0)$ as the telescoping sum:
\[Z(t)-Z(t_0)=\sum_{l=1}^{\infty}Z(\pi_l(t))-Z(\pi_{l-1}(t))
=\sum_{l=1}^{\infty}\frac1m\sum_{i=1}^m\varepsilon_i\widetilde{q}_i\dotp{\widetilde{U}_i}{\pi_l(t)-\pi_{l-1}(t)}.\]
We claim that the telescoping sum converges with probability 1 for any $t\in T$. 
Indeed, note that for each fixed set of realizations of $\{\mathbf{x}_i\}_{i=1}^m$ and $\{\varepsilon_i\}_{i=1}^m$,
each summand is bounded as
\[
|\varepsilon_i \widetilde{q}_i \langle\widetilde{U}_i,\pi_l(t)-\pi_{l-1}(t)\rangle |
\leq|\widetilde{q}_i|\|\widetilde{U}_i\|_2\|\pi_l(t)-\pi_{l-1}(t)\|_2
\leq|\widetilde{q}_i|\|\widetilde{U}_i\|_2(\|\pi_l(t)-t\|_2+\|\pi_{l-1}(t)-t\|_2).
\]
Furthermore, since $T$ is a compact subset of $\mathbb{R}^d$, its Gaussian mean width is finite. 
Thus, by lemma \ref{mmt}, $\gamma_2(T)\leq L\omega(T)<\infty$. This inequality further implies that the sum on the left hand side of \eqref{inter-1} converges with probability 1. 
%Since the factor $2^l$ is at least one, the sum without this factor must also converge.

Next, with $\beta\geq 8$ being fixed, we split the index set $\{l\geq1\}$ into the following three subsets: 
\begin{align*}
I_1&=\{l\geq1:2^{l}\beta<\log em\};\\
I_2&=\{l\geq1:\log em\leq 2^{l}\beta< m\};\\
I_3&=\{l\geq1:2^{l}\beta\geq m\}.
\end{align*}
By the assumptions in Theorem \ref{master-bound} and the bound $\beta\geq8$, we have that $m\geq(\omega(T)+1)^2\beta^2\geq64$, implying that 
$\log em=1+\log m <m$, and hence these three index sets are well defined. 
Depending on $\beta$, some of them might be empty, but this only simplifies our argument by making the partial sum over such an index set equal 0. 

The following argument yields a bound for $Z(\pi_l(t))-Z(\pi_{l-1}(t))$, assuming all three index sets are nonempty. 
Specifically, we show that
\begin{equation}
\label{chaining-goal}
\mb P \left(\sup_{t\in T}\left|\sum_{l\in I_j}\left(Z(\pi_l(t))-Z(\pi_{l-1}(t))\right)\right|
\geq C\frac{\gamma_2(T)\beta}{\sqrt{m}} \right)
\leq ce^{-\beta/2},
\end{equation}
for $C=C(\kappa,\phi)$ and $j=1,2,3$, respectively.

 %#############
\subsubsection{The case $l\in I_1$.}\label{first-chaining}
 %#############
\begin{proof}[Proof of inequality \eqref{chaining-goal} for the index set $I_1$]
Recall that $\tau=m^{\frac{1}{2(1+\kappa)}}$. \\
For each $t\in T$ we apply Bernstein's inequality (Lemma \ref{Bernstein}) to estimate each summand 
\[
Z(\pi_l(t))-Z(\pi_{l-1}(t))=\frac1m\sum_{i=1}^m\varepsilon_i\widetilde{q}_i\dotp{\widetilde{U}_i}{\pi_l(t)-\pi_{l-1}(t)}.
\] 
For any integer $p\geq2$, we have the following chains of inequalities:%\lcomm{for better readability, subscripts in `generic' variables that appear without summation in front, such as those in the display, have been removed}
%
%We have
%%\marginpar{\Stas{Below, $2+\eps$ moments for $q$ is sufficient}}
%\begin{align*}
%&\expect{(\varepsilon_i\langle\widetilde{U}_i,\pi_l(t)-\pi_{l-1}(t)\rangle \widetilde{q}_i)^2}\\
%\leq&\expect{(\langle\widetilde{U}_i,\pi_l(t)-\pi_{l-1}(t)\rangle q_i)^2}\\
%\leq&\expect{\langle\widetilde{U}_i,\pi_l(t)-\pi_{l-1}(t)\rangle^4}^{1/2}\expect{ q_i^4}^{1/2}\\
%\leq&4\|\widetilde{U}_i\|_{\psi_2}^2\|\pi_l(t)-\pi_{l-1}(t)\|_2^2\phi^{1/2},
%\end{align*}
%where the second from last inequality follows from Cauchy-Schwartz and the last inequality follows from
%Lemma \ref{lemma:ball} that $\expect{|\langle\widetilde{U}_i,t\rangle|^4}^{1/4}
%\leq2\|\widetilde{U}_i\|_{\psi_2}\|t\|_2,~\forall t\in\mathbb{R}^d$ as well as the fourth moment assumption. We define the constant $C_1:=4\|\widetilde{U}_i\|_{\psi_2}^2\phi^{1/2}$. Furthermore, 
%for any integer $p\geq2$, we have 
\begin{align*}
&\expect{\left|\varepsilon\widetilde{q}\dotp{\widetilde{U}}{\pi_l(t)-\pi_{l-1}(t)} \right|^p}\\
\leq&\expect{\left|\varepsilon\dotp{\widetilde{U}}{\pi_l(t)-\pi_{l-1}(t)} \right|^p q^{2}
	\cdot|\widetilde{q}|^{p-2}}\\
\leq&\expect{\left|\dotp{\widetilde{U}}{\pi_l(t)-\pi_{l-1}(t)} \right|^p q^2}
\cdot \tau^{p-2}\\
\leq& \tau^{p-2}
\expect{ \left|\dotp{\widetilde{U}}{\pi_l(t)-\pi_{l-1}(t)}\right|
	^{\frac{1+\kappa}{\kappa}p}}^{\frac{\kappa}{1+\kappa}}\expect{ q^{2(1+\kappa)}}^{\frac{1}{1+\kappa}}\\
\leq&  \tau^{p-2}\|\widetilde{U}\|_{\psi_2}^p\left(\frac{(1+\kappa)p}{\kappa}\right)^{p/2}\phi^{\frac{1}{1+\kappa}}
\|\pi_l(t)-\pi_{l-1}(t)\|_2^p,
\end{align*}
where the second inequality follows from the truncation bound, the third from H\"{o}lder's inequality, and the last from the assumption that 
$\expect{q^{2(1+\kappa)}}\leq\phi$ and the following bound: by Lemma \ref{lemma:ball}, $\widetilde{U}_i$ is sub-Gaussian, hence for any $p\geq 2$
\[
\l(\mb E\dotp{\widetilde{U}_i}{\mathbf{v}}^{\frac{1+\kappa}{\kappa}p} \r)^{\frac{\kappa}{(1+\kappa)p}}
\leq\left(\frac{(1+\kappa)p}{\kappa}\right)^{1/2}\|\widetilde{U}_i\|_{\psi_2}\|\mathbf{v}\|_2,~\forall \mathbf{v}\in\mathbb{R}^d.
\]
We also note that $\|\widetilde{U}_i\|_{\psi_2}$ does not depend on $d$ by Lemma \ref{lemma:ball}. 
Next, by Stirling's approximation, 
$p!\geq\sqrt{2\pi}\sqrt{p}(p/e)^p$, thus there exist constants
$C'=C'(\kappa,\phi)$ and $C''=C''(\kappa)$ such that
%\lcomm{these constants appear to depend on more than just $\kappa$, for instance, also on $\phi$, on $p$, and the sub-Gaussian norm of $\widetilde{U}$; further changes so form agrees with lemma invoked.}\xcomm{I removed the sentence ``depend only on $\kappa$''}
\begin{align*}
\mb E\left|\varepsilon \widetilde{q} \dotp{\widetilde{U}}{\pi_l(t)-\pi_{l-1}(t)} \right|^p
\leq \frac{p!}{2}C'\|\pi_l(t)-\pi_{l-1}(t)\|_2^2(C''\tau\|\pi_l(t)-\pi_{l-1}(t)\|_2)^{p-2}.
\end{align*}
Bernstein's inequality (Lemma \ref{Bernstein}), with $\sigma=C'\|\pi_l(t)-\pi_{l-1}(t)\|_2$, $D=C''\tau\|\pi_l(t)-\pi_{l-1}(t)\|_2$ with $\tau=m^{1/2(1+\kappa)}$ now implies
\begin{align*}
\mb P \left(
\left|\frac1m\sum_{i=1}^m\varepsilon_i\widetilde{q}_i\dotp{\widetilde{U}_i}{\pi_l(t)-\pi_{l-1}(t)}\right|\geq\left(\frac{C'\sqrt{2u}}{\sqrt{m}}+\frac{C'' u}{m^{1-\frac{1}{2(1+\kappa)}}}\right)\|\pi_l(t)-\pi_{l-1}(t)\|_2
\right)
\leq 2e^{-u},
\end{align*}
for any $u>0$. Taking $u=2^l\beta$,  noting that as $\beta\geq8$ by assumption, 
we have $m\geq(\omega(T)+1)^2\beta^2\geq64$, and since $l \in I_1$, $2^l\leq2^l\beta < \log em$. 
In turn, this implies
\begin{align*}
\frac{2^l}{m^{1-\frac{1}{2(1+\kappa)}}}=\frac{2^{l/2}}{m^{1/2}}\cdot\frac{2^{l/2}}{m^{\kappa/2(1+\kappa)}}
\leq\frac{2^{l/2}}{m^{1/2}}\cdot\sqrt{\frac{\log em}{m^{\kappa/(1+\kappa)}}}
\leq\sqrt{\frac{1+\kappa}{\kappa}}\frac{2^{l/2}}{m^{1/2}},
\end{align*}
where the last inequality follows from the fact that $\log em$ is dominated by $\frac{1+\kappa}{\kappa}m^{\kappa/(1+\kappa)}$ for all $m \ge 1$.
This inequality implies that there exists a positive constant $C=C(\kappa,\phi)$ such that for any $\beta \geq 8$
\begin{align}\label{inter-2}
%Pr\left[\left|\frac1m\sum_{i=1}^m\varepsilon_i\widetilde{q}_i\dotp{\widetilde{U}_i}{\pi_l(t)-\pi_{l-1}(t)}\right|\geq C\frac{2^{l/2}\beta}{\sqrt{m}}\|\pi_l(t)-\pi_{l-1}(t)\|_2\right]
\mb P \l( \Omega_{l,t} \r) \leq2\exp(-2^l\beta),
\end{align}
where for all $l \ge 1$ and $t \in T$ we let 
%\lcomm{reordered to avoid repetition of long expression}
\begin{align*}
\Omega_{l,t}=\left\{\omega: \left|\frac1m\sum_{i=1}^m\varepsilon_i\widetilde{q}_i\dotp{\widetilde{U}_i}{\pi_l(t)-\pi_{l-1}(t)}\right|\geq C\frac{2^{l/2}\beta}{\sqrt{m}}\|\pi_l(t)-\pi_{l-1}(t)\|_2\right\}.
\end{align*}
Notice that for each $l \ge 1$ the number of pairs $(\pi_l(t),\pi_{l-1}(t))$ appearing in the sum in \eqref{chaining-goal} can be bounded by 
$|\mathcal{A}_l|\cdot|\mathcal{A}_{l-1}|\leq2^{2^{l+1}}$. Thus, by a union bound and \eqref{inter-2},
\begin{align*}
\mb P\left(\bigcup_{t\in T}\Omega_{l,t}\right)
\leq \
2\cdot 2^{2^{l+1}}\exp(-2^l\beta),
\end{align*}
and hence,
\begin{align*}
\mb P \l( \bigcup_{l\in I_1,t\in T}\Omega_{l,t} \r) \leq
&\sum_{l\in I_1}2\cdot 2^{2^{l+1}}\exp(-2^l\beta)\\
\leq&\sum_{l\in I_1}2\cdot 2^{2^{l+1}}\exp\left(-2^{l-1}\beta-\beta/2\right)
\leq ce^{-\beta/2},
\end{align*}
for some absolute constant $c>0$, where in the last inequality we use the fact $\beta\geq 8$ to get a geometrically decreasing sequence. 
Thus, on the complement of the event $\cup_{l\in I_1,t\in T}\Omega_{l,t}$, we have that with probability at least $1-ce^{-\beta/2}$,
\begin{align*}
\sup_{t\in T}\left|\sum_{l\in I_1}\left(Z(\pi_l(t))-Z(\pi_{l-1}(t))\right)\right|
\leq&\sup_{t\in T}\sum_{l\in I_1}\left|Z(\pi_l(t))-Z(\pi_{l-1}(t))\right|\\
\leq&\sup_{t\in T}C\sum_{l\in I_1}\frac{2^{l/2}\beta}{\sqrt{m}}\|\pi_l(t)-\pi_{l-1}(t)\|_2\\
\leq&\sup_{t\in T}C\sum_{l=1}^{\infty}\frac{2^{l/2}\beta}{\sqrt{m}}\|\pi_l(t)-\pi_{l-1}(t)\|_2\\
\leq&4C\frac{\gamma_2(T)\beta}{\sqrt{m}},
\end{align*}
for $C=C(\kappa,\phi)$, where the last inequality follows from triangle inequality $\|\pi_l(t)-\pi_{l-1}(t)\|_2\leq\|\pi_{l-1}(t)-t\|_2+\|\pi_l(t)-t\|_2$ and \eqref{inter-1}. 
This proves the inequality \eqref{chaining-goal} for $l\in I_1$.
\end{proof}

%##########################
\subsubsection{The case $l\in I_2$.}
%##########################

%\marginpar{\Stas{At first look, 3 moments for $q$ should be enough here is well, but it needs to be checked more carefully.}}

This is the most technically involved case of the three. 
For any fixed $t\in T$ and $l\in I_2$, we let $X_i=\widetilde{q}_i \dotp{\widetilde{U}_i}{\pi_l(t)-\pi_{l-1}(t)} $ and $w_i=\langle\widetilde{U}_i,\pi_l(t)-\pi_{l-1}(t)\rangle$. Then $X_i=\widetilde{q}_i w_i$ and
\begin{equation}
\label{inter-average}
Z(\pi_l(t))-Z(\pi_{l-1}(t))=\frac1m\sum_{i=1}^m\varepsilon_iX_i
=\frac1m\sum_{i=1}^m\varepsilon_iw_i\widetilde{q}_i.
\end{equation}
For every fixed $k\in\{1,2,\cdots,m-1\}$ and fixed $u>0$, we bound the summation using the following inequality
%\lcomm{need to state the conditions required on $X_i,i=1,\ldots,m$ for this inequality to be applied, and show that such conditions are satisfied in this case}
\begin{align*}
\mb P\left(
\left|\sum_{i=1}^m\varepsilon_iX_i\right|\geq\sum_{i=1}^kX^*_i+u\left(\sum_{i=k+1}^m(X_i^*)^2\right)^{1/2}
\right)
\leq2\exp(-u^2/2),
\end{align*}
where $\{X_i^*\}_{i=1}^m$ is the \textit{non-increasing} rearrangement of $\{|X_i|\}_{i=1}^m$ and $\{\varepsilon_i\}_{i=1}^m$ is a sequence of i.i.d. Rademancher random variables independent of $\{X_i\}_{i=1}^m$.
\begin{remark}
This bound was first stated and proved in \cite{Rademancher-sums} with a sequence of fixed constants $\{X_i\}_{i=1}^m$. The current form can be obtained using independence property and conditioning on $\{X_i\}_{i=1}^m$. 
Furthermore, \cite{Rademancher-sums} tells us that the optimal choice of $k$ is at $\mathcal{O}(u^2)$
Applications of this inequality to generic chaining-type arguments were previously introduced by \cite{Mendelson-2}.
\end{remark}
Letting $J$ be the set of indices of the variables corresponding to the $k$ largest coordinates of $\{|w_i|\}_{i=1}^m$ and of $\{|\widetilde{q}_i|\}_{i=1}^m$, we have $|J|\leq2k$ and with probability at least $1 - 2\exp(-u^2/2)$
%\lcomm{seems like there should be a $2$ in front of $\exp(-u^2/2)$ in order to agree with the inequality just introduced},
\begin{align}
\left|\sum_{i=1}^m\varepsilon_iX_i\right|
&\leq\sum_{i\in J}X^*_i+u\left(\sum_{i\in J^c}(X_i^*)^2\right)^{1/2}\nonumber\\
&\leq2\sum_{i=1}^kw_i^*\widetilde{q}_i^*+u\left(\sum_{i\in J^c}(w_i^*\widetilde{q}_i^*)^2\right)^{1/2}\nonumber\\
%&\Stas{ \text{What if we use H\"{o}lder's instead of C-S for the second term?}} \\
&\leq2\left(\sum_{i=1}^k(w_i^*)^2\right)^{1/2}\left(\sum_{i=1}^k(\widetilde{q}_i^*)^2\right)^{1/2}
+u\left(\sum_{i= k+1}^m(w_i^*)^{\frac{2(1+\kappa)}{\kappa}}\right)^{\frac{\kappa}{2(1+\kappa)}}\left(\sum_{i=k+1}^m(\widetilde{q}_i^*)^{2(1+\kappa)}\right)^{\frac{1}{2(1+\kappa)}}\nonumber\\
&\leq2\left(\sum_{i=1}^k(w_i^*)^2\right)^{1/2}\left(\sum_{i=1}^m\widetilde{q}_i^2\right)^{1/2}
+u\left(\sum_{i= k+1}^m(w_i^*)^{\frac{2(1+\kappa)}{\kappa}}\right)^{\frac{\kappa}{2(1+\kappa)}}\left(\sum_{i=1}^m\widetilde{q}_i^{2(1+\kappa)}\right)^{\frac{1}{2(1+\kappa)}}\label{inter-3}
\end{align}
where the second to last inequality is a consequence of H\"{o}lder's inequality. 
We take $u=2^{(l+1)/2}\sqrt{\beta}$. The key is to pick an appropriate cut point $k$ for each $l\in I_2$. Here, we choose $k=\lfloor2^l\beta/\log(em/2^l\beta)\rfloor$, which makes $k=\mathcal{O}(2^l\beta)$ and also guarantees that $k\in\{1,2,\cdots,m-1\}$; see Lemma \ref{support-2}. Under this choice, we have the following lemma:

\begin{lemma}\label{cut-point-1}
Let $k=\lfloor2^l\beta/\log(em/2^l\beta)\rfloor$, $w_i=\dotp{\widetilde{U}_i}{\pi_l(t)-\pi_{l-1}(t)}$ and $\{w_i^*\}_{i=1}^m$ be the nonincreasing rearrangement of $\{|w_i|\}_{i=1}^m$. Then there exists an absolute constant $C>1$ such that for all $\beta \ge 8$,
%\lcomm{small change to make the order of quantifiers clearer}
\[
\mb P \l( 
\left(\sum_{i=1}^k(w_i^*)^2\right)^{1/2}\geq C2^{l/2}\|\pi_l(t)-\pi_{l-1}(t)\|_2\sqrt{\beta}
\right)
\leq2\exp(-2^{l}\beta).
\]
\end{lemma}
\begin{proof}
By Lemma \ref{lemma:ball}, we know that $\{w_i\}_{i=1}^m$ are i.i.d. sub-Gaussian random variables. 
Thus, by Lemma \ref{prop-1}, $w_i^2$ is sub-exponential with norm 
\begin{equation}\label{inter-subexp}
\|w_i^2\|_{\psi_1}=2\|w_i\|_{\psi_2}^2 \leq 2\|\widetilde{U}_i\|_{\psi_2}^2\|\pi_l(t)-\pi_{l-1}(t)\|_2^2.
\end{equation}
It then follows from Bernstein's inequality (Lemma \ref{Bernstein}) that for any fixed set $J\subseteq \{1,2,\cdots,m\}$ with $|J|=k$,
\begin{align*}
\mb P \left(
\left| \frac{1}{k} \sum_{i\in J} \left(w_i^2-\expect{w_i^2}\right) \right| 
\geq 2\| \widetilde{U}_i\|_{\psi_2}^2\|\pi_l(t)-\pi_{l-1}(t) \|_2^2
\left(\sqrt{\frac{2u}{k}}+\frac{u}{k}\right)
\right)
\leq2\exp(-u).
\end{align*}
We choose $u=4\cdot2^{l}\beta=2^{l+2}\beta$. 
Since $2^l\beta\geq \lfloor2^l\beta/\log(em/2^l\beta)\rfloor=k\geq1$, the factor $u/k$
%\lcomm{previously claimed that the other, smaller factor `dominates'} 
dominates the right hand side. 
Noting that $\expect{w_i^2}=\|\pi_l(t)-\pi_{l-1}(t)\|_2^2$, we obtain
\begin{align*}
\mb P \l(
\left(\sum_{i\in J}w_i^2\right)^{1/2}\geq C2^{l/2}\|\pi_l(t)-\pi_{l-1}(t)\|_2\sqrt{\beta}
\right)
\leq2\exp(-4\cdot2^l\beta),
\end{align*}
where $C\leq 4\| \widetilde{U}_i\|_{\psi_2}$; note that the upper bound for $C$ is independent of $d$ by Lemma \ref{ball-lemma-0}. 
Thus, 
\begin{align*}
&\mb P \left(
\left(\sum_{i = 1}^k(w_i^*)^2\right)^{1/2}\geq C2^{l/2}\|\pi_l(t)-\pi_{l-1}(t)\|_2\sqrt{\beta}
\right)
\\
=& \mb P \left(
\exists J\subseteq\{1,\cdots,m\},~|J|=k:
\left(\sum_{i\in J}w_i^2\right)^{1/2}\geq C2^{l/2}\|\pi_l(t)-\pi_{l-1}(t)\|_2\sqrt{\beta}
\right)
\\
\leq&
\left(
\begin{array}{c}
 m   \\
 k   
\end{array}
\right)\cdot 
\mb P\left(
\left(\sum_{i\in J}w_i^2\right)^{1/2}\geq C2^{l/2}\|\pi_l(t)-\pi_{l-1}(t)\|_2\sqrt{\beta}
\right) 
\\
\leq&2\left(
\begin{array}{c}
 m   \\
 k   
\end{array}
\right)\exp(-4\cdot2^l\beta)\\
\leq&2\left(\frac{em}{k}\right)^k\exp(-4\cdot2^l\beta)
\leq2\exp(-2^l\beta),
%\leq&2\exp\left(\frac{2^l\beta}{\log\frac{em}{2^l\beta}}\log\left(\frac{em}{2^l\beta}\log\frac{em}{2^l\beta}\right)
%-3\cdot2^l\beta\right)\leq2\exp(2^{l+1}\beta-4\cdot2^l\beta)\leq2\exp(2^l\beta),
\end{align*}
where the last step follows from $\left(\frac{em}{k}\right)^k\leq\exp(3\cdot2^l\beta)$, an inequality proved in lemma \ref{support-1} in Appendix \ref{app-A}.
\end{proof}

\begin{lemma}\label{cut-point-2}
Let $k=\lfloor2^l\beta/\log(em/2^l\beta)\rfloor$, $w_i=\dotp{\widetilde{U}_i}{\pi_l(t)-\pi_{l-1}(t)}$ and $\{w_i^*\}_{i=1}^m$ be the non-increasing rearrangement of $\{|w_i|\}_{i=1}^m$. 
Then
\[
\mb P \left(
\left(\sum_{i=k+1}^m(w_i^*)^{\frac{2(1+\kappa)}{\kappa}}\right)^{\frac{\kappa}{2(1+\kappa)}}\geq C(\kappa) m^{\frac{\kappa}{2(1+\kappa)}}\|\pi_l(t)-\pi_{l-1}(t)\|_2
\right)
\leq\exp(-2^{l}\beta),
\]
for any $\beta\geq8$ and some constant $C(\kappa)>0$.
\end{lemma}
\begin{proof}
To avoid possible confusion, we use $i$ to index the nonincreasing rearrangement and $j$ for the original sequence.
We start by noting that $\{w_j\}_{j=1}^m$ are i.i.d. sub-Gaussian random variables with $\|w_j\|_{\psi_2} \le \|\widetilde{U}_j\|_{\psi_2}\|\pi_l(t)-\pi_{l-1}(t)\|_2$. By an equivalent definition of sub-Gaussian random variables (Lemma 5.5. of \cite{introduction-to-random-matrix}), we have for any fixed $j\in\{1,2,\ldots,m\}$,
\bea \label{eq:boundw_j.change.u}
\mb P \l(
|w_j|-\expect{|w_j|}\geq Cu\|\widetilde{U}_j\|_{\psi_2}\|\pi_l(t)-\pi_{l-1}(t)\|_2
\right)
\leq e^{-u^2},
\ena
for any $u>0$ and an absolute constant $C>0$.

To establish the claim of the lemma, we bound each
%\lcomm{removed the word `rearrangement' here, as the entire sequence is rearranged} 
$w_i^*$ separately for $i=1,2\ldots,m$ and then combine individual bounds. 
Instead of using a fixed value of $u$ in \eqref{eq:boundw_j.change.u}, our choice of $u$ will depend on the index $i$. 
Specifically, for each $w_i^*$, we choose 
%\lcomm{removed `we choose the bound' as it is $u$ being chosen, not the (probability) right hand side} 
$u=c_\kappa(m/i)^{\kappa/4(1+\kappa)}$ with
\bea \label{def:ckappa}
c_\kappa:=\max\left\{\frac{\sqrt{5}\left(2+\frac{4}{\kappa}\right)^{\frac{2+\kappa}{4(1+\kappa)}}}{e^{1/2(1+\kappa)}},\sqrt{\frac{4(1+\kappa)}{\kappa}}\right\}.
\ena
The reason for this choice will be clear as we proceed. 

First, for a fixed nonincreasing rearrangement index $i>k$, 
by \eqref{eq:boundw_j.change.u} and the fact that 
\[
\expect{|w_j|}\leq\expect{w_j^2}^{1/2}=\|\pi_l(t)-\pi_{l-1}(t)\|_2,~\forall j\in\{1,2,\cdots,m\},
\]
we have
\begin{align*}
\mb P \left(
|w_j|\geq \left(1+Cc_\kappa\|\widetilde{U}_j\|_{\psi_2}\right)\left(\frac{m}{i}\right)^{\frac{\kappa}{4(1+\kappa)}}\|\pi_l(t)-\pi_{l-1}(t)\|_2
\right)
\leq \exp\left(-c_\kappa^2\left(\frac{m}{i}\right)^{\frac{\kappa}{2(1+\kappa)}}\right),&\\
\forall j\in\{1,2,\cdots,m\}.&
\end{align*}
To simplify notation, let $C'=1+Cc_\kappa\|\widetilde{U}_j\|_{\psi_2}$ (note that it depends only on $\kappa$). 
It then follows that 
\begin{align*}
& \mb P \left(
w_{i}^*\geq C'\left(\frac{m}{i}\right)^{\frac{\kappa}{4(1+\kappa)}}\|\pi_l(t)-\pi_{l-1}(t)\|_2
\right) \\
=& \mb P\left(
\exists J\subseteq\{1,\cdots,m\},~|J|=i:~w_j\geq C'\left(\frac{m}{i}\right)^{\frac{\kappa}{4(1+\kappa)}}\|\pi_l(t)-\pi_{l-1}(t)\|_2,~\forall j\in J
\right) \\
\leq&\left(
\begin{array}{c}
 m   \\
 i   
\end{array}
\right)
\mb P\left(
|w_j|\geq C'\left(\frac{m}{i}\right)^{\frac{\kappa}{4(1+\kappa)}}\|\pi_l(t)-\pi_{l-1}(t)\|_2\right)^i
\\
\leq&\left(
\begin{array}{c}
 m   \\
 i   
\end{array}
\right)\exp\left(- c^2m^{\frac{\kappa}{2(1+\kappa)}}i^{\frac{2+\kappa}{2(1+\kappa)}}\right)\\
\leq&\left(\frac{em}{i}\right)^i
\exp\left(- c^2m^{\frac{\kappa}{2(1+\kappa)}}i^{\frac{2+\kappa}{2(1+\kappa)}}\right).
%\leq\left(\frac{em}{i}\right)^i\exp(- 9m^{1/4}i^{3/4}).
\end{align*}
By a union bound, we have
\begin{align*}
& \mb P\left(
\exists i>k: w_{i}^*\geq C'\left(\frac{m}{i}\right)^{\frac{\kappa}{4(1+\kappa)}}\|\pi_l(t)-\pi_{l-1}(t)\|_2
\right) 
\\
\leq&\sum_{i=k+1}^m\left(\frac{em}{i}\right)^i
\exp\left(- c^2m^{\frac{\kappa}{2(1+\kappa)}}i^{\frac{2+\kappa}{2(1+\kappa)}}\right)\\ =
&\sum_{i=k+1}^m\exp\left(i\log\left(\frac{em}{i}\right) 
- c^2m^{\frac{\kappa}{2(1+\kappa)}}i^{\frac{2+\kappa}{2(1+\kappa)}}\right)\\
\leq&m\cdot\exp\left(k\log\left(\frac{em}{k}\right)
- c^2m^{\frac{\kappa}{2(1+\kappa)}}k^{\frac{2+\kappa}{2(1+\kappa)}}\right)\\
\leq&\exp\left(4\cdot 2^l\beta- c^2m^{\frac{\kappa}{2(1+\kappa)}}k^{\frac{2+\kappa}{2(1+\kappa)}}\right),
\end{align*}
where the second to last inequality follows since by the definition \eqref{def:ckappa} of $c_\kappa$, $c_\kappa\geq\sqrt{4(1+\kappa)/\kappa}$, the function 
$v(i)=i\log\left(\frac{em}{i}\right) 
-c_\kappa^2m^{\frac{\kappa}{2(1+\kappa)}}\cdot i^{\frac{2+\kappa}{2(1+\kappa)}}$ 
is monotonically decreasing with respect to $i$ (recall that $i\leq m$), and thus is dominated by $v(k)$.
%\lcomm{isn't really the `first term' since the sum starts at $k+1$} 
The final inequality follows from Lemma \ref{support-1} as well as the fact that 
$\log m \le \log(em)\leq 2^l\beta$. 
Furthermore, by Lemma \ref{support-2} in the Appendix \ref{app-A} and \eqref{def:ckappa} implying $c_\kappa\geq\sqrt{5}\left(2+\frac{4}{\kappa}\right)^{\frac{2+\kappa}{4(1+\kappa)}}/e^{1/2(1+\kappa)}$, we have
\[
c_\kappa^2m^{\frac{\kappa}{2(1+\kappa)}}k^{\frac{2+\kappa}{2(1+\kappa)}}
\geq 5\cdot2^l\beta.
\]
Overall, we have the following bound:
\begin{align*}
\mb P\left[\exists i>k: w_{i}^*\geq C'\left(\frac{m}{i}\right)^{\frac{\kappa}{4(1+\kappa)}}\|\pi_l(t)-\pi_{l-1}(t)\|_2\right]
\leq\exp\left(4\cdot2^l\beta-5\cdot2^l\beta\right)
\leq\exp(-2^l\beta).
%=&\sum_{i=k+1}^m\exp\left(i\left(\log\left(\frac{em}{i}\right)-\beta (\frac{m}{i})^{1/4}\right)\right)\\
%\leq&C_1\exp\left(k\left(\log\left(\frac{em}{k}\right)-\beta (\frac{m}{k})^{1/4}\right)\right),
\end{align*}
Thus, with probability at least $1-\exp(-2^{l}\beta)$,
\[
w_{i}^*\leq C'\left(\frac{m}{i}\right)^{\frac{\kappa}{4(1+\kappa)}}\|\pi_l(t)-\pi_{l-1}(t)\|_2,~\forall i>k,
\]
hence with the same probability
\begin{align*}
\left(\sum_{i=k+1}^m(w_i^*)^{\frac{2(1+\kappa)}{\kappa}}\right)^{\frac{\kappa}{2(1+\kappa)}}
\leq&C'\|\pi_l(t)-\pi_{l-1}(t)\|_2\left(\sum_{i=k+1}\left(\frac{m}{i}\right)^{1/2}\right)^{\frac{\kappa}{2(1+\kappa)}}\\
\leq&C'\|\pi_l(t)-\pi_{l-1}(t)\|_2m^{\frac{\kappa}{4(1+\kappa)}}\left(\int_1^m\frac{dx}{x^{1/2}}\right)^{\frac{\kappa}{2(1+\kappa)}}\\
\leq&2^{\frac{\kappa}{2(1+\kappa)}}C'\|\pi_l(t)-\pi_{l-1}(t)\|_2m^{\frac{\kappa}{2(1+\kappa)}},
\end{align*}
and the desired result follows.
\end{proof}

\begin{lemma}\label{cut-point-3}
The following inequalities hold for any $\beta\geq 8$:
\begin{align*}
&\mb P\left(
\left(\sum_{i=1}^m\widetilde{q}_i^2\right)^{1/2}\geq C'\sqrt{\beta m}
\right)
\leq 2e^{-\beta}, 
\\
& \mb P\left(
\left(\sum_{i=1}^m\widetilde{q}_i^{2(1+\kappa)}\right)^{\frac{1}{2(1+\kappa)}}\geq C''(\beta m)^{\frac{1}{2(1+\kappa)}}
\right)
\leq2e^{-\beta},
\end{align*}
for some positive constants $C'=C'(\phi,\kappa), \ C''=C''(\phi,\kappa)$.
\end{lemma}
\begin{proof}
Recall that $\widetilde{q}_i=\textrm{sign}(q_i)(|q_i|\wedge\tau)$, 
$\tau=m^{1/2(1+\kappa)}$, and $\phi=\expect{q_i^{2(1+\kappa)}}$. 
Thus,
$\expect{\widetilde{q}_i^2}\leq\expect{q_i^2}\leq\phi^{1/1+\kappa}$,
and for any integer $p\geq2$, we have
\[
\expect{\widetilde{q}_i^{2p}} = \expect{\widetilde{q}_i^{2p-2(1+\kappa)}\widetilde{q}_i^{2(1+\kappa)}}
\leq m^{\frac{p-1-\kappa}{1+\kappa}}\expect{q_i^{2(1+\kappa)}}
\leq m^{\frac{p-1-\kappa}{1+\kappa}}\phi.
\]
Thus, for any $p\geq2$,
\[
\expect{|\widetilde{q}_i^2-\expect{\widetilde{q}_i^2}|^{p}}
\leq
\expect{\widetilde{q}_i^{2p}}+\left(\expect{q_i^2} \right)^p
\leq 
m^{\frac{p-1-\kappa}{1+\kappa}}\phi+\phi^{\frac{p}{1+\kappa}}
\leq 
(m+\phi)^{\frac{1-\kappa}{1+\kappa}}\phi(m+\phi)^{\frac{p-2}{1+\kappa}}.
\]
By Bernstein's inequality (Lemma \ref{Bernstein}), with probability at least $1-2e^{-\beta}$,
\begin{align*}
\left|\frac1m\sum_{i=1}^m\widetilde{q}_i^2-\expect{\widetilde{q}_i^2}\right|
&\leq\left(\frac{\sqrt{2\beta} (m+\phi)^{\frac{1-\kappa}{2(1+\kappa)}}\phi^{1/2}}{m^{1/2}}
+\frac{\beta (m+\phi)^{\frac{1}{1+\kappa}}}{m}\right)\\
&\leq\frac{\sqrt{2\beta} (1+\phi)^{\frac{1-\kappa}{2(1+\kappa)}}\phi^{1/2}+\beta (1+\phi)^{\frac{1}{1+\kappa}}}{m^{\frac{\kappa}{1+\kappa}}},
\end{align*}
which implies the first claim. 
To establish the second claim, note that for any $p \geq 2$,
%\lcomm{what inequality is being applied here? Seems to be something like $|a-b|^p \le |a|^p+|b|^p$, which is not valid. Note that one also needs to take absolute value to take non-integer powers.}
\begin{align*}
\mb E\left|\widetilde{q}_i^{2(1+\kappa)}-\expect{\widetilde{q}_i^{2(1+\kappa)}}\right|^p
\leq &
C(p)\l( 
\mb E \l| \widetilde{q}_i^{2(1+\kappa)p} \r| + \l(\mb E \l| q_i^{2(1+\kappa)} \r|\r)^p \r) \\
\leq
&C(p) \l( \mb E\l| \widetilde{q}_i^{2(1+\kappa)(p-1)}q_i^{2(1+\kappa)} \r| + \phi^p\r) \\
\leq& C(p) (m^{p-1}\phi+\phi^p) \leq C(p)(m+\phi)^{p-2}(m+\phi)\phi,
\end{align*}
where we used the fact that $|\widetilde{q_i} |\leq m^{1/2(1+\kappa)}$ to obtain the third inequality. 
Bernstein's inequality implies that with probability at least $1-2e^{-\beta}$,
\[
\left|\frac1m\sum_{i=1}^m\widetilde{q}_i^{2(1+\kappa)}-\expect{\widetilde{q}_i^{2(1+\kappa)}}\right|
\leq\sqrt{2\beta}(1+\phi)\phi^{1/2}+\beta(1+\phi),
\]
which yields the second part of the claim.
\end{proof}

%\lcomm{appears some proof continues, perhaps best to place inside a formal proof, with statement, or just some direction given here, so the reader understands the intention here}
%\xcomm{I put them in a proof of \eqref{chaining-goal}.}
\begin{proof}[Proof of inequality \eqref{chaining-goal} for the index set $I_2$]
Combining Lemmas 
\ref{cut-point-1} and \ref{cut-point-2} with the inequality \eqref{inter-3}, and setting $u=2^{l/2}\sqrt{\beta}$, 
we get that with probability at least $1-4\exp(-2^l\beta)$, for all $l \in I_2$, 
%\lcomm{3 may need to be replaced by 4, see comment above}
\begin{align*}
|Z(\pi_l(t))- &Z(\pi_{l-1}(t))|\leq 
\\
&
C\|\pi_l(t)-\pi_{l-1}(t)\|_2\frac{2^{l/2}\sqrt{\beta}}{m}\left(\left(\sum_{i=1}^m\widetilde{q}_i^2\right)^{1/2}
+m^{\frac{\kappa}{2(1+\kappa)}}\left(\sum_{i=1}^m\widetilde{q}_i^{2(1+\kappa)}\right)^{\frac{1}{2(1+\kappa)}}
\right), 
\end{align*}
for some constant $C=C(\kappa,\phi)>0$; note that the factor $1/m$ appears due to equality \eqref{inter-average}. 
Next, we apply a chaining argument similar to the one used in Section \ref{first-chaining}, we obtain that with probability at least $1-ce^{-\beta/2}$,
\begin{align}\label{inter-4}
\sup_{t\in T}\left|\sum_{l\in I_2}\left(Z(\pi_l(t))-Z(\pi_{l-1}(t))\right)\right|
\leq C\frac{\gamma_2(T)\sqrt{\beta}}{m}\left(\left(\sum_{i=1}^m\widetilde{q}_i^2\right)^{1/2}
+m^{\frac{\kappa}{2(1+\kappa)}}\left(\sum_{i=1}^m\widetilde{q}_i^{2(1+\kappa)}\right)^{\frac{1}{2(1+\kappa)}}\right),
\end{align}
 for a positive constant $C=C(\kappa,\phi)$ and an absolute constant $c>0$. 
In order to handle the remaining terms involving $\widetilde{q}_i$ in \eqref{inter-4}, we apply 
 Lemma \ref{cut-point-3}, which gives
\begin{align*}
\sup_{t\in T}\left|\sum_{l\in I_2}\left(Z(\pi_l(t))-Z(\pi_{l-1}(t))\right)\right|
\leq C\frac{\gamma_2(T)\beta}{\sqrt{m}},
\end{align*}
with probability at least $1-ce^{-\beta/2}$, where $C=C(\kappa,\phi)$ and $c>0$ are positive constants and $\beta\geq8$. 
This completes the second part of the chaining argument.
\end{proof}

%
%
%%First, we can compute the subexponential norm of $\widetilde{q}_i^2$ and $\widetilde{q}_i^4$ respectively, which result in $\|\widetilde{q}_i^2\|_{\psi_1}\leq m^{1/2}$ and $\|\widetilde{q}_i^4\|_{\psi_1}\leq m$. By Bernstein's inequality (Lemma \ref{Bernstein}), we have
%\Stas{I didn't understand the explanation with $\psi_1$ - norm -- probably, it is easier to apply ``standard'' Bernstein's inequality noting that 
% $\sqrt{\mb E\widetilde q_i^8}=\sqrt{\mb E\widetilde q_i^4\widetilde q_i^4}\leq \sqrt{m}\sqrt{\mb E\widetilde q_i^4}$.}
%\begin{align*}
%&Pr\left[\left|\frac1m\sum_{i=1}^m\widetilde{q}_i^2-\expect{\widetilde{q}_i^2}\right|
%\geq\left(\sqrt{\frac{2\beta }{m^{1/2}}}+\frac{\beta}{m^{1/2}}\right)\right]\leq2e^{-\beta}\\
%&Pr\left[\left|\frac1m\sum_{i=1}^m\widetilde{q}_i^4-\expect{\widetilde{q}_i^4}\right|
%\geq\left(\sqrt{2\beta} +\beta\right)\right]\leq2e^{-\beta}.
%\end{align*}
%By the fourth moment assumption $\expect{q_i^4}\leq\phi<\infty$, we have $\expect{\widetilde{q}_i^4}\leq\phi$ and $\expect{\widetilde{q}_i^2}\leq\sqrt{\phi}$. Thus,
%\begin{align*}
%&Pr\left[\sum_{i=1}^m\widetilde{q}_i^2
%\geq m\left(\sqrt{\phi}+\sqrt{\frac{2\beta }{m^{1/2}}}+\frac{\beta}{m^{1/2}}\right)\right]\leq2e^{-\beta},\\
%&Pr\left[\sum_{i=1}^m\widetilde{q}_i^4
%\geq m\left(\phi+\sqrt{2\beta} +\beta\right)\right]\leq2e^{-\beta},
%\end{align*} 
%which implies the claim by $\beta\geq8$.

%################
\subsubsection{The case $l\in I_3$.}
%################
\begin{proof}[Proof of inequality \eqref{chaining-goal} for the index set $I_3$]
%In this case, there is no need to make the type of splitting as in \eqref{inter-3}.
Direct application of Cauchy-Schwartz on \eqref{inter-average} yields, for all $t\in T$,
\[
|Z(\pi_l(t))-Z(\pi_{l-1}(t))|\leq\left(\frac1m\sum_{i=1}^mw_i^2\right)^{1/2}\left(\frac1m\sum_{i=1}^m\widetilde{q}_i^2\right)^{1/2},
\]
where $w_i=\dotp{\widetilde{U}_i}{\pi_l(t)-\pi_{l-1}(t)}$ are sub-Gaussian random variables. 
Thus, by Lemma \ref{prop-1}, $\omega_i^2$ are sub-exponential with norm bounded as in \eqref{inter-subexp}. 
Using Bernstein's inequality again, we deduce that  
\begin{align*}
\mb P \left(
\left|\frac1m\sum_{i=1}^m\left(w_i^2-\expect{w_i^2}\right)\right|\geq2\|\widetilde{U}_i\|_{\psi_2}^2\|\pi_l(t)-\pi_{l-1}(t)\|_2^2
\left(\sqrt{\frac{2u}{m}}+\frac{u}{m}\right)
\right)
\leq2\exp(-u).
\end{align*}
Let $u=2^{l}\beta$. 
Using the fact that $2^l\beta/m\geq 1$ as well as $\expect{w_i^2}=\|\pi_l(t)-\pi_{l-1}(t)\|_2^2$, we see that the term $u/m$ dominates the right hand side and
\begin{align*}
\mb P \left(
\left(\frac1m\sum_{i=1}^mw_i^2\right)^{1/2}\geq C\|\pi_l(t)-\pi_{l-1}(t)\|_2\frac{2^{l/2}\sqrt{\beta}}{\sqrt{m}}
\right)
\leq2\exp(-2^l\beta),
\end{align*}
for some absolute constant $C>0$. 
Thus, repeating a chaining argument of section \ref{first-chaining} (namely, the argument following \eqref{inter-2}), we obtain
\begin{align*}
\sup_{t\in T}\left|\sum_{l\in I_3}\left(Z(\pi_l(t))-Z(\pi_{l-1}(t))\right)\right|
\leq C\frac{\gamma_2(T)\sqrt{\beta}}{\sqrt{m}}\left(\frac1m\sum_{i=1}^m\widetilde{q}_i^2\right)^{1/2}
\end{align*}
with probability at least $1-ce^{-\beta/2}$ for some absolute constants $C,c>0$. 
Combining this inequality with the first claim of Lemma \ref{cut-point-3} gives
\begin{align*}
\sup_{t\in T}\left|\sum_{l\in I_3}\left(Z(\pi_l(t))-Z(\pi_{l-1}(t))\right)\right|
\leq C\frac{\gamma_2(T)\beta}{\sqrt{m}},
\end{align*}
with probability at least $1-ce^{-\beta/2}$ for absolute constants $C,c>0$ and any $\beta\geq 8$. 
This finishes the bound for the third (and final) segment of the ``chain''.
\end{proof}

%##########################################################
\subsubsection{Finishing the proof of Lemma \ref{concentration-multiplier}}
%##########################################################

\begin{proof}
So far, we have shown that
\begin{align}
\sup_{t\in T}\left|Z(t)-Z(t_0)\right|
=&\sup_{t\in T}\left|\sum_{l\geq1}\left(Z(\pi_l(t))-Z(\pi_{l-1}(t))\right)\right|\nonumber\\
\leq& \sum_{j\in\{1,2,3\}}\sup_{t\in T}\left|\sum_{l\in I_j}\left(Z(\pi_l(t))-Z(\pi_{l-1}(t))\right)\right|\nonumber\\
\leq&C\frac{\gamma_2(T)\beta}{\sqrt{m}}, \label{final-inter}
\end{align}
with probability at least $1-ce^{-\beta/2}$ for some positive constants $C=C(\kappa,\phi)$ and $c$, and any $\beta\geq8$. 
To finish the proof, it remains to bound 
$|Z(t_0)|=\left|\frac1m\sum_{i=1}^m\varepsilon_i\widetilde{q}_i
\dotp{\widetilde{U}_i}{t_0} \right|$. 
With $\Delta_d(T)$ defined in \eqref{def:Deltad}, and since $t_0$ is an arbitrary point in $T$, we trivially have $\|t_0\|_2\leq\Delta_d(T)$. 
Applying Bernstein's inequality in a way similar to Section \ref{first-chaining} yields
\begin{align*}
\mb P \left(
\left|\frac1m\sum_{i=1}^m\varepsilon_i\widetilde{q}_i\dotp{\widetilde{U}_i}{t_0} \right|\geq\left(\frac{C'\sqrt{2u}}{\sqrt{m}}+\frac{C'' u}{m^{1-\frac{1}{2(1+\kappa)}}}\right)\Delta_d(T)
\right)
\leq 2e^{-u},
\end{align*}
for some constants $C'=C'(\kappa,\phi), \ C''=C''(\kappa,\phi)>0$ and any $u>0$. 
Choosing $u=\beta$ gives 
\begin{align*}
\mb P \left(
\left|\frac1m\sum_{i=1}^m\varepsilon_i\widetilde{q}_i\dotp{\widetilde{U}_i}{t_0} \right|\geq 
\frac{C\Delta_d(T)\beta}{\sqrt{m}}
\right)
\leq 2e^{-\beta},
\end{align*}
for a constant $C=C(\kappa,\phi)>0$ and any $\beta\geq 0$.
%\lcomm{as $C$ is not specified, it seems it is not required that $\beta \ge 8$ here.} 
Combining this bound with \eqref{final-inter} shows that with probability at least $1-ce^{-\beta/2}$, 
\[
\sup_{t\in T}\left|\frac1m\sum_{i=1}^m\varepsilon_i\langle\widetilde{U}_i,t\rangle \widetilde{q}_i\right|
\leq C\frac{(\gamma_2(T)+\Delta_d(T))\beta}{\sqrt{m}}
\leq C\frac{(L\omega(T)+\Delta_d(T))\beta}{\sqrt{m}},
\]
for $C=C(\kappa,\phi)$, an absolute constant $L>0$ and all $\beta\geq8$; note that the last inequality follows from Lemma \ref{mmt}. 
We have established \eqref{major-criterion}, thus completing the proof.
\end{proof}

%\bibliographystyle{unsrt}
%\bibliography{sigproc}
%\bibliographystyle{alpha}
\bibliographystyle{imsart-nameyear}
\bibliography{bibliography,bibliography2,bibliography3}

\appendix
\section{Technical results.}
\label{app-A}

\begin{lemma}
\label{basic-inequality}
For any nonnegative random variable $X$, if 
$\mb P \l( X>K\beta \r) \leq ce^{-\beta/2}$ for some constants $K,c>0$ and all $\beta\geq\beta_0\geq 0$, then,
\[
\expect{X}\leq K\left(\beta_0+2ce^{-\beta_0/2}\right).\]
\end{lemma}
\begin{proof}
Using a well known identity for the expectation of non-negative random variables,
\begin{align*}
\expect{X}=&\int_0^{\infty}\mb P\l(X>u\r)du
=K\int_0^{\infty}\mb P\l( X>K\beta \r)d\beta\\
\leq&K \left( \beta_0+\int_{\beta_0}^{\infty}\mb P\l( X>K\beta \r) d\beta \right)
\leq K\left( \beta_0+\int_{\beta_0}^{\infty}ce^{-\beta/2}d\beta\right)\\
=&K\left(\beta_0+2ce^{-\beta_0/2}\right).
\end{align*}
\end{proof}

%\begin{equation}\label{major-criterion}
%Pr\left[\sup_{t\in T}\left|Z(t)\right|
%\geq C\frac{(\omega(D(\Theta,\eta\theta_*)\cap\mathbb{S}^{d-1})+\Delta_d(T))\beta}{\sqrt{m}}\right]
%\leq ce^{-\beta/2}.
%\end{equation}

\begin{lemma} 
\label{prop-1}
If $X$ and $Y$ are sub-Gaussian random variables, then the product $XY$ is a subexponential random variable, and
\[
%\|XY\|_{\psi_1}\leq 2\|X\|_{\psi_2}\|Y\|_{\psi_2}.
\|XY\|_{\psi_1}\leq \|X\|_{\psi_2}\|Y\|_{\psi_2}.
\]
\end{lemma}
\begin{proof}
See \citep{wellner1}. 

	\begin{comment}

By definition $\|XY\|_{\psi_1}=\sup_{p\geq1}p^{-1}\expect{|XY|^p}^{1/p}$. Applying $2ab\leq a^2+b^2$ and the triangle inequality, for any $\epsilon>0$,
\begin{align*}
\expect{|XY|^p}^{1/p}\leq\expect{\left|\frac{X^2}{2\epsilon}+\frac{\epsilon Y^2}{2}\right|^p}^{1/p}
\leq\frac{1}{2\epsilon}\expect{X^{2p}}^{1/p}+\frac\epsilon2\expect{Y^{2p}}^{1/p}.
\end{align*}
Multiplying both sides by $p^{-1}$ and taking supremum, we see that this inequality implies
\[\|XY\|_{\psi_1}\leq\frac{1}{2\epsilon}\|X^2\|_{\psi_1}+\frac{\epsilon}{2}\|Y^2\|_{\psi_1}.\]
Notice that $\|X^2\|_{\psi_1}$ can be bounded as
\begin{align*}
\|X^2\|_{\psi_1}=&\sup_{p\geq1}\left(p^{-1/2}\expect{X^{2p}}^{1/2p}\right)^2
=\sqrt2\sup_{p\geq1}\left((2p)^{-1/2}\expect{X^{2p}}^{1/2p}\right)^2
\leq2\|X\|_{\psi_2}^2.
\end{align*}
Thus,
\[\|XY\|_{\psi_1}\leq\frac{1}{\epsilon}\|X\|^2_{\psi_2}+\epsilon\|Y\|^2_{\psi_2}.\]
Finally,  choosing $\epsilon=\|X\|_{\psi_2}/\|Y\|_{\psi_2}$ finishes the proof.

	\end{comment}
	
\end{proof}

\begin{lemma}\label{support-1}
Let $k=\lfloor2^l\beta/\log(em/2^l\beta)\rfloor$ and $l\in I_2$, then,
$\left(\frac{em}{k}\right)^k\leq\exp(3\cdot2^l\beta).$
\end{lemma}
\begin{proof}
If $k\geq2$, then, $2^l\beta/\log(em/2^l\beta)\geq2$, which implies $2^l\beta\geq2\log(em/2^l\beta)$. Thus,
\begin{align*}
\left(\frac{em}{k}\right)^k
\leq&2\exp\left(\frac{2^l\beta}{\log\frac{em}{2^l\beta}}\log\left(\frac{em}{\frac{2^l\beta}{\log\frac{em}{2^l\beta}}-1}\right)\right)\\
\leq&2\exp\left(\frac{2^l\beta}{\log\frac{em}{2^l\beta}}\log\left(\frac{em}{2^l\beta-\log\frac{em}{2^l\beta}}\log\frac{em}{2^l\beta}\right)\right)\\
\leq&2\exp\left(\frac{2^l\beta}{\log\frac{em}{2^l\beta}}\log\left(\frac{2em}{2^l\beta}\log\frac{em}{2^l\beta}\right)\right)\leq\exp(3\cdot2^l\beta),
\end{align*}
where the second from last inequality follows from $\left(\frac{em}{k}\right)^k\leq\exp(3\cdot2^l\beta)$, and the last inequality follows from $m\geq2^l\beta$, thus, $\log(2em/2^l\beta)/\log(em/2^l\beta)\leq2$.

On the other hand, if $k=1$, then, since $\log em\leq2^l\beta$,
$
\left(\frac{em}{k}\right)^k=em=\exp(\log em)\leq\exp(2^l\beta),
$
finishing the proof. 
\end{proof}

\begin{lemma}
\label{support-2}
With $m \ge 1, \beta \ge 1, \kappa \in (1,0)$ and 
$l\in I_2=\{l\geq1:\log em\leq 2^{l}\beta< m\}$, the integer
$k=\lfloor2^l\beta/\log(em/2^l\beta)\rfloor$ satisfies $k \ge 1$, and
\[
\frac{\left(2+\frac{4}{\kappa}\right)^{\frac{2+\kappa}{2(1+\kappa)}}}{e^{1/(1+\kappa)}}m^{\frac{\kappa}{2(1+\kappa)}}k^{\frac{2+\kappa}{2(1+\kappa)}}\geq 2^l\beta.
\] \end{lemma}

\begin{proof}
Since $2^l \beta\geq\log(em) \ge 1$, 
it follows that $k\geq1$, and thus $k\geq2^l\beta/2\log(em/2^l\beta)$. It is then enough to show that
\[
\frac{\left(1 + \frac{2}{\kappa}\right)^{\frac{2+\kappa}{2(1+\kappa)}}}{e^{1/(1+\kappa)}}
\left(\frac{m}{2^l\beta}\right)^{\frac{\kappa}{2(1+\kappa)}}
\geq
\left(\log\frac{em}{2^l\beta}\right)^{\frac{2+\kappa}{2(1+\kappa)}}.
\]
Raising both sides to the power of $2(1+\kappa)/\kappa$, equivalently
\[
\left.\left(
1+\frac2\kappa\right)^{\frac{2+\kappa}{\kappa}}\right/e^{\frac2\kappa}
\geq
\left.\left(
\log\frac{em}{2^l\beta}\right)^{\frac{2+\kappa}{\kappa}}\right/\frac{m}{2^l\beta}.
\] 
Consider the function $g(x)=\left(\log ex\right)^{\frac{2+\kappa}{\kappa}}/x$. Note that as $m>2^l\beta$, to prove the inequality above it suffices to show that the $\sup_{x \ge 1}g(x)$ is upper bounded by the left hand side. Taking the derivative of $g(x)$ yields
\[g'(x)=\frac{\frac{2+\kappa}{\kappa}(1+\log x)^{2/\kappa}-(1+\log x)^{(2+\kappa)/\kappa}}{x^2}.\]
Since $x\geq1$, the only critical point at which the global maximum occurs is given by $x=e^{2/\kappa}$. 
As $g\left(e^{2/\kappa}\right)$ is exactly equal to the left hand side the proof is complete.
\end{proof}

%###########################################################
\section{Decomposable norms and Restricted Compatibility.}
%###########################################################
\label{app-B}

In this section, we recall some facts about decomposable norms that have been introduced in \cite{general-m-estimator}. 
%$\|\cdot\|_{\mathcal{K}}$ is decomposable. 
\begin{definition}
\label{def:decomposable}
Suppose that $\mathcal{L}\subseteq \m L_1$ are two subspace of $\mathbb{R}^d$, and let $\mathcal{L}_1^{\perp}$ be the orthogonal complement of $\mathcal{L}_1$. 
Norm $\|\cdot\|_{\mathcal{K}}$ is said to be decomposable with respect to $(\mathcal{L},~\mathcal{L}_1^\perp)$ if for any $\theta\in\mathbb{R}^d$,
\begin{align*}
\|\theta_1+\theta_2\|_{\mathcal{K}}=\| \Pi_{\mathcal{L}}\theta_1\|_{\mathcal{K}}+\| \Pi_{\mathcal{L}_1^{\perp}}\theta\|_{\mathcal{K}},
\end{align*}
where $\Pi_{\mathcal{L}}$ and $\Pi_{\mathcal{L}_1^{\perp}}$ stand for the orthogonal projectors onto $\mathcal{L}$ and $\mathcal{L}_1^\perp$ respectively.
\end{definition}
%\xcolor{We say that a vector $\mathbf{x}\in\mathbb{R}^d$ is group sparse if its entries are divided into $N\leq d$ groups $\{\mathbf{x}_1,\mathbf{x}_2,\cdots,\mathbf{x}_N\}$ and only $s$ out of these $N$ subvectors are not all zero. The corresponding $\ell_{2,1}$-norm is defined as $\|\mathbf{x}\|_{2,1}=\sum_{i=1}^N\|\mathbf{x}_i\|_2$.}

It is well known that many frequently used norms, including the $\ell_1$ norm of a vector and the nuclear norm of a matrix, are decomposable with respect to the appropriately chosen pair of subspaces. 
For instance, the $\ell_1$ norm is decomposable with respect to the pair of subspaces $(\m L(J), \m L(J)^\perp)$, where 
\begin{align}
\label{eq:L(J)}
&
\m L(J):=\l\{ v\in \mb R^d: \ v_j=0 \text{ for all } j\notin J\r\}
\end{align} 
consists of sparse vectors with non-zero coordinates indexed by a set $J\subseteq \l\{ 1,\ldots, d\r\}$. 

Let $W_1\subseteq \mb R^{d_1}, \ W_2\subseteq \mb R^{d_2}$ be two linear subspaces. 
Then we define the subspace $\m L(W_1,W_2)\subseteq \mb R^{d_1\times d_2}$ via 
\begin{align*}
&
\m L(W_1,W_2):= \l\{ M\in \mb R^{d_1\times d_2}: \ \mathrm{row}(M)\subseteq W_1, \ \mathrm{col}(M)\subseteq W_2  \r\},
\end{align*}
where $\mathrm{row}(M)$ and $\mathrm{col}(M)$ are the linear subspaces spanned by the rows and columns of $M$ respectively, and 
\begin{align}
\label{eq:L_12}
&
\m L_1^\perp(W_1,W_2):= \l\{ M\in \mb R^{d_1\times d_2}: \ \mathrm{row}(M)\subseteq W_1^\perp, \ \mathrm{col}(M)\subseteq W_2^\perp  \r\}.
\end{align}
Then the nuclear norm $\|\cdot\|_\ast$ is decomposable with respect to $\l( \m L(W_1,W_2),\m L_1^\perp(W_1,W_2 )\r)$ (see \citep{general-m-estimator} for details). 

%Furthermore, we will get a simpler characterization of the restricted set $\mb{S}_{c_0}(\theta)$ if $\theta\in\mathcal{S}$ and 
%the norm $\|\cdot\|_{\mathcal{K}}$ is decomposable with respect to $(\mathcal{S},\mathcal{S}^{\perp})$. 
Assume that the norm $\|\cdot\|_\m K$ is decomposable with respect to $(\m L,\m L_1^\perp)$, and let $\theta\in \m L$. 
It is clear that for any $\mathbf{v}\in\mb{S}_{c_0}(\theta)$
\begin{equation}
\label{app-1}
\|\theta+\mathbf{v}\|_{\mathcal{K}}=
\| \Pi_{\mathcal{L}}\theta+\Pi_{\mathcal{L}_1}\mathbf{v}+\Pi_{\mathcal{L}_1^{\perp}}\mathbf{v}\|_{\mathcal{K}}
\leq \| \Pi_{\mathcal{L}}\theta\|_{\mathcal{K}}+\frac{1}{c_0}\| \Pi_{\mathcal{L}_1}\mathbf{v} \|_{\m K}+ \| \Pi_{\mathcal{L}_1^{\perp}}\mathbf{v}\|_{\mathcal{K}}.
\end{equation}
%By the definition of the restricted set \eqref{eq:rest-set}, for any $\mathbf{v}\in\mb{S}_{c_0}(\theta)$, 
%$
%\|\theta+\mathbf{v}\|_{\mathcal{K}}
%\leq \|\theta\|_{\mathcal{K}}+\frac{1}{c_0}\|\mathbf{v}\|_{\mathcal{K}}.
%$
Since $\theta\in\mathcal{L}$, decomposability and the triangle inequality imply that 
\begin{align*}
\| \Pi_{\mathcal{L}}\theta+\Pi_{\mathcal{L}_1}\mathbf{v}+\Pi_{\mathcal{L}_1^{\perp}}\mathbf{v}\|_{\mathcal{K}}
&
=\|\Pi_{\mathcal{L}}\theta+\Pi_{\mathcal{L}_1}\mathbf{v}\|_{\mathcal{K}} 
+\|\Pi_{\mathcal{L}_1^{\perp}}\mathbf{v}\|_{\mathcal{K}} \\
&
\geq\|\Pi_{\mathcal{L}}\theta\|_{\mathcal{K}}-\|\Pi_{\mathcal{L}_1}\mathbf{v}\|_{\mathcal{K}}
+\|\Pi_{\mathcal{L}_1^{\perp}}\mathbf{v}\|_{\mathcal{K}}.
\end{align*}
%and
%\[
%\frac{1}{c_0}\| \Pi_{\mathcal{L}}\mathbf{v}+ \Pi_{\mathcal{L}^{\perp}}\mathbf{v}\|_{\mathcal{K}}
%\leq\frac{1}{c_0}\| \Pi_{\mathcal{S}}\mathbf{v}\|_{\mathcal{K}}+\frac{1}{c_0}\| \Pi_{\mathcal{L}^{\perp}}\mathbf{v}\|_{\mathcal{K}}.
%\]
Substituting this bound into \eqref{app-1} gives
\begin{align*}
-\| \Pi_{\mathcal{L}}\mathbf{v}\|_{\mathcal{K}}+\| \Pi_{\mathcal{L}_1^{\perp}}\mathbf{v}\|_{\mathcal{K}}
\leq\frac{1}{c_0}\| \Pi_{\mathcal{L}_1}\mathbf{v}\|_{\mathcal{K}}+\frac{1}{c_0}\| \Pi_{\mathcal{L}_1^{\perp}}\mathbf{v}\|_{\mathcal{K}},
\end{align*}
which implies that for any $\mathbf{v}\in\mb{S}_{c_0}(\theta)$
\[
\| \Pi_{\mathcal{L}_1^{\perp}}\mathbf{v}\|_{\mathcal{K}} \leq 
\frac{c_0+1}{c_0-1}\| \Pi_{\mathcal{L}_1}\mathbf{v}\|_{\mathcal{K}}.
\]
It is easy to see that the set of all $\mf{v}$ satisfying the inequality above is a convex cone, which we will denote by $C_{c_0}=C_{c_0}(\m K)$. 
Since ${\mathbb{S}}_{c_0}(\theta)\subseteq C_{c_0}$,  
\[
\Psi\left(\mb{S}_{c_0}(\theta)\right)\leq\Psi\left(C_{c_0}\right)
\]
by definition of the restricted compatibility constant. 
This inequality is useful due to the fact that it is often easier to estimate $\Psi\left(C_{c_0}\right)$. 

Finally, we make a remark that is useful when dealing with non-isotropic measurements. 
Let $\mf\Sigma\succ 0$ be a $d\times d$ matrix, and consider the norm corresponding to the convex set $\mf \Sigma^{1/2}\m K$, so that 
$\|\mf v\|_{\mf \Sigma^{1/2}\m K}=\|\mf\Sigma^{-1/2}\mf v\|_\m K$. 
It is easy to see that $C_{c_0}(\mf \Sigma^{1/2}\m K) = \mf\Sigma^{1/2}C_{c_0}(\m K)$, hence 
\begin{align*}
\Psi \l( C_{c_0}(\mf \Sigma^{1/2}\m K); \mf\Sigma^{1/2} \m K \r) & = 
\sup_{\mf v\in \mf\Sigma^{1/2}\m K \setminus \{0\}}\frac{\| \mf v\|_{\mf\Sigma^{1/2}\m K}}{\|\mf v\|_2} = 
\sup_{\mf u\in \m K\setminus \{0\}}\frac{\|\mf u\|_\m K}{\| \mf \Sigma^{1/2}\mf u\|_2} \\
&
\leq \| \mf\Sigma^{-1/2}\| \, \Psi \l( C_{c_0}(\m K); \m K \r). 
\end{align*}
\textbf{Example 1: $\ell_1$ norm. }
Let $\m L(J)$ be as in \eqref{eq:L(J)} with $|J|=s\leq d$. 
If $v\in \mb R^d$ belongs to the corresponding cone $C(c_0)$, then clearly $\|v\|_1\leq \frac{2c_0}{c_0-1}\|v_J\|_1$, where $v_J:=\Pi_{\m L(J)} v$. 
Hence 
\[
\|v\|_1 \leq \frac{2c_0}{c_0-1}\|v_J\|_1\leq \frac{2c_0}{c_0-1}\sqrt{|J|}\|v\|_2,
\] 
and $\Psi(C_{c_0})\leq \frac{2c_0}{c_0-1}\sqrt{s}.$
\\
\textbf{Example 2: nuclear norm.}
Let $\m L_1^\perp(W_1,W_2)$ be as in \eqref{eq:L_12}. 
Note that for any $v\in \mb R^{d_1\times d_2}$, 
$\Pi_{\m L_1^\perp(W_1,W_2)}v = \Pi_{W_2^\perp} v \Pi_{W_1^\perp}$, where $\Pi_{W_1^\perp}$ and $\Pi_{W_2^\perp}$ are the orthogonal projectors onto subspaces $W_1\subseteq \mb R^{d_1}$ and $W_2\subseteq \mb R^{d_2}$ respectively. 
Then for any $v\in C_{c_0}$, we have that 
\begin{align}
\label{eq:rank}
\|v\|_\ast \leq \| \Pi_{\m L_1^\perp(W_1,W_2)}v \|_\ast + \| \Pi_{\m L_1(W_1,W_2)}v \|_\ast \leq \frac{2c_0}{c_0-1}\| \Pi_{\m L_1(W_1,W_2)}v \|_\ast.
\end{align}
Note that 
\begin{align*}
& 
\Pi_{\m L_1(W_1,W_2)}v = v - \Pi_{W_2^\perp} v \Pi_{W_1^\perp} =  \Pi_{W_2^\perp} v \Pi_{W_1} + \Pi_{W_2}v, 
\end{align*}
hence $\mathrm{rank}\l(  \Pi_{\m L_1(W_1,W_2)}v \r)\leq 2\max\l(\dim(W_1),\dim(W_2)\r)$, which yields together with \eqref{eq:rank} that
\[
\|v\|_\ast \leq \frac{2c_0}{c_0-1}\| \Pi_{\m L_1(W_1,W_2)}v \|_\ast 
\leq \frac{2c_0}{c_0-1}\sqrt{2\max\l(\dim(W_1),\dim(W_2)\r)}\|v\|_2, 
\]
and 
$\Psi(C_{c_0})\leq \frac{2\sqrt{2}c_0}{c_0-1}\sqrt{\max\l(\dim(W_1),\dim(W_2)\r)}.$

\end{document}